\documentclass[a4paper,11pt]{amsart}

\usepackage[centertags]{amsmath}

\usepackage{amscd}

\usepackage{amssymb}

\usepackage{fullpage}
\usepackage{color}
\usepackage{euscript}
\usepackage{tikz}
\newlength{\defbaselineskip} 
\usepackage{footnote}
\usepackage{color}
\usepackage{todonotes}
\usepackage{url}

\usepackage[all,cmtip]{xy}

\newtheorem{thm}{Theorem}[section]
\newtheorem{cor}[thm]{Corollary}
\newtheorem{corollary}[thm]{Corollary}
\newtheorem{corr}[thm]{Corollary}

\newtheorem{lem}[thm]{Lemma}
\newtheorem{prop}[thm]{Proposition}
\newtheorem{conj}[thm]{Conjecture}
\newtheorem{prob}[thm]{Problem}

\theoremstyle{definition}
\newtheorem{defi}[thm]{Definition}

\newtheorem{rem}[thm]{Remark}

\usetikzlibrary{matrix,arrows}
\makeatletter \tikzset{
  edge node/.code={%
      \expandafter\def\expandafter\tikz@tonodes\expandafter{\tikz@tonodes #1}}}
\makeatother \tikzset{
  subseteq/.style={
    draw=none,
    edge node={node [sloped, allow upside down, auto=false]{$\subseteq$}}},
  Subseteq/.style={
    draw=none,
    every to/.append style={
      edge node={node [sloped, allow upside down, auto=false]{$\subseteq$}}}
  }
}

 \numberwithin{equation}{section}
\numberwithin{equation}{section} \theoremstyle{definition}
\DeclareMathOperator{\Pic}{Pic}
\DeclareMathOperator{\NS}{NS}
\DeclareMathOperator{\Aut}{Aut}
\DeclareMathOperator{\Fix}{Fix}

\DeclareMathOperator{\Sym}{Sym}

\DeclareMathOperator{\im}{im} 
\DeclareMathOperator{\id}{id}

          \newcommand\PP{{\mathbb{P}}}
           
          \newcommand\C{{\mathbb{C}  }}

            \newcommand\ZZ{{\mathbb Z}}

          \newcommand\oo{\mathcal O}
             \newcommand\Q{\mathbb Q}
             \newcommand{\ra}{\rightarrow}
          \newcommand\Z{\mathbb{Z}}

\definecolor{zielony}{rgb}{0.5, 0.9, 0.1}
\definecolor{czerwony}{rgb}{0.8, 0.2, 0.1}
\definecolor{niebieski}{rgb}{0.3, 0.1, 0.9}

\newcounter{appendice}

\begin{document}
\title{Projective models of Nikulin orbifolds}
\author[C.~Camere]{Chiara Camere}
\address{Chiara Camere, Universit\`a degli Studi di Milano,
Dipartimento di Matematica, Via Cesare Saldini 50, 20133 Milano,
Italy} \email{chiara.camere@unimi.it}
\urladdr{http://www.mat.unimi.it/users/camere/en/index.html}

\author[A.~Garbagnati]{Alice Garbagnati}
\address{Alice Garbagnati, Universit\`a degli Studi di Milano, Dipartimento di Matematica, via Cesare Saldini 50 20133 Milano, Italy }
\email{alice.garbagnati@unimi.it}
\urladdr{https://sites.google.com/site/alicegarbagnati/}

\author[G.~Kapustka]{Grzegorz Kapustka}
\address{Grzegorz Kapustka, Jagiellonian University Cracow, ul. \L{}ojasiewicza 6, 30-348 Krak\'o{}w}
\email{grzegorz.kapustka@uj.edu.pl}
 
 \author[M.~Kapustka]{Micha\l{} Kapustka}
\address{Micha\l{} Kapustka, Institute of Matematics of the Polish academy of Sciences, ul. \'Sniadeckich 8, 00-656 Warsaw}
\address{Micha\l{} Kapustka, University of Stavanger, Department of Mathematics and Natural Sciences, NO-4036 Stavanger, Norway}
\email{michal.kapustka@impan.pl}
\urladdr{}

\begin{abstract}
We study projective fourfolds of $K3^{[2]}$-type with a symplectic involution and the deformations of their quotients, called orbifolds of Nikulin types; they are IHS orbifolds.
 We compute the Riemann--Roch formula for Weil divisors on such orbifolds 
 and describe the first complete family of orbifolds of Nikulin type with a polarization of degree $2$ as double covers of special complete intersections $(3,4)$ in $\PP^6$.

\end{abstract}
\maketitle

\section{Introduction}
An orbifold $Y$ is said irreducible holomorphic symplectic if $Y\setminus\mathrm{Sing}(Y)$ is simply connected and admits a unique, up to a scalar multiple, non-degenerate holomorphic $2$-form.
Such manifolds are intensively studied \cite{BL,Menet,MT,FM} because they can be seen as a natural generalization of smooth irreducible hyper-K\"ahler manifolds.
There are only a few known families of these orbifolds, see \cite{MenetTorelli, Perego}.
A first non-smooth example of irreducible holomorphic symplectic orbifold of dimension $4$ is given as deformation of a partial resolution of the quotient of a fourfold of $K3^{[2]}$-type by a symplectic involution: we call this an {\it orbifold of Nikulin type}, in analogy with Nikulin surfaces in dimension two; the first examples were studied by \cite{MT}.

The aims of this paper are to find the relation between families of projective $K3^{[2]}$-type fourfolds $X$ with a symplectic involution $\sigma$ and families of Nikulin orbifolds $Y$ which are partial resolutions of $X/\sigma$  and  to describe geometrical properties of orbifolds of Nikulin type. In the first part we describe projective fourfolds of $K3^{[2]}$-type
having a symplectic involution both giving a lattice theoretic classification of them and providing a geometric interpretation of general members of their families. 
Note that it is known \cite{Cam12, Mo} that the quotient $X/\sigma$ is singular at $28$ points of type $\frac{1}{2}(1,1,1,1)$ and a $K3$ surface with transversal $A_1$ singularities. The general deformations of the partial resolution $Y$ of such a quotient are singular only at the $28$ points.
In particular, they are not factorial but $2$-factorial. 
In the second part we use the results from the previous sections to
find the Riemann--Roch formula for Weil divisors on orbifolds of Nikulin type, which depends on their Beauville--Bogomolov--Fujiki $q$ (BBF for short in the rest of the paper) degrees and the number of singular points in which the Weil divisor fails to be Cartier.
As a result, we describe in the third part the first complete family of 
projective irreducible holomorphic symplectic orbifolds as double covers of special complete intersections $(3,4)$ in $\PP^6$. Note that there are only 
six known such families, all for smooth hyper-K\"ahler manifolds (see the survey \cite[Section 3.6.1]{Debarre} and references therein).

Let us describe each of the three parts separately.

Symplectic involutions on smooth $K3$ surfaces are nowadays well
understood thanks to foundational work of Nikulin
\cite{Nikulin-symplectic} and then of van Geemen and Sarti
\cite{vGS, FV}. These involutions have $8$ fixed points and the quotient admits a resolution of singularities being a $K3$ surface with Picard rank $\geq 8$, called a Nikulin surface.
The first aim of this paper, described in Section \ref{sec: 4-fold with involution}, is to look at symplectic involutions
$\sigma$ on projective fourfolds of $K3^{[2]}$-type of degree
$2d$. We describe their possible Picard lattices and transcendental groups; as a consequence we identify their families in terms of lattice polarized families of $K3^{[2]}$-type fourfolds. We prove the following result  (see Table \ref{table: j,NS,T}), which is the analogue of the result \cite[Proposition 2.2]{vGS} for surfaces with symplectic involution.
\begin{thm}
Let $X$ be a generic projective fourfold of $K3^{[2]}$-type admitting a symplectic involution, then the pair $(T_X,\NS(X))$ of the transcendental lattice and the N\'eron--Severi group of $X$ is one of the following:\\
$(U^{\oplus 2}\oplus E_8(-2)\oplus \langle -2d\rangle\oplus \langle -2\rangle$, $\Lambda_{2d}$);\\
$(U^{\oplus 2}\oplus D_4(-1)\oplus \langle -2d\rangle\oplus \langle -2\rangle^{\oplus 5},\Lambda_{2d})$, with $d\equiv 1\mod 2$;\\ $(U^{\oplus 2}\oplus E_8(-2)\oplus K_d,\Lambda_{2d})$ with $d\equiv 3\mod 4$;\\ $(U^{\oplus 2}\oplus D_4(-1)\oplus \langle -2d\rangle\oplus \langle -2\rangle^{\oplus 5},\widetilde{\Lambda}_{2d})$ with $d\equiv 0\mod 2$,\\
where the lattices involved are defined in the notation  in Section \ref{subsec: 4folds with inv. lattices} and $d$ is a positive integer.

Vice versa if $X$ is a generic projective fourfold of $K3^{[2]}$-type such that $\NS(X)$ is isometric either to $\Lambda_{2d}$ or to $\widetilde{\Lambda}_{2d}$, then it admits a symplectic involution.
\end{thm}

In Section \ref{subsec: 4folds inv. models} we show that the general member of these families can be described either as Hilbert scheme of two points on a K3 surface or as moduli space of (possibly twisted) sheaves on a K3 surface, see Table \ref{table: proj models X}. In both  cases the K3 surfaces involved lie in 12-dimensional families of lattice polarized K3 surfaces and are resolution of singularities of a K3 surface with $7$ nodes.

In Section \ref{sec: Nikulin fourfold} we consider the quotient $X/\sigma$ and its partial resolution $Y$. The knowledge of the N\'eron--Severi group and of the transcendental group of $X$ allows one to compute the ones of $Y$ and thus the family of fourfolds of $K3^{[2]}$-type $X$ determines the family of the Nikulin orbifolds $Y$. In particular we prove the following result(see Table \ref{table: NS and T of Y}), which is the analogue of the result \cite[Corollary 2.2]{GS} for Nikulin surfaces.

\begin{thm}
	Let $X$ be a fourfold of $K3^{[2]}$-type admitting a symplectic involution $\sigma$, and $Y$ be the Nikulin orbifold partial resolution of $X/\sigma$ then\\
	$T_X\simeq U^{\oplus 2}\oplus E_8(-2)\oplus \langle -2d\rangle\oplus \langle -2\rangle$ if and only if $T_Y=U(2)^{\oplus 2}\oplus E_8(-1)\oplus \langle -4d\rangle\oplus \langle -4\rangle$;\\
	$T_X\simeq U^{\oplus 2}\oplus D_4(-1)\oplus \langle -2d\rangle\oplus \langle -2\rangle^{\oplus 5}$ if and only if $T_Y\simeq U(2)^{\oplus 2}\oplus E_7(-1)\oplus K_d(2) \langle -2\rangle$;\\ $T_X\simeq U^{\oplus 2}\oplus E_8(-2)\oplus K_d$ if and only if $T_Y\simeq U(2)^{\oplus 2}\oplus K_d(2)\oplus E_8(-1)$;\\ $T_X\simeq U^{\oplus 2}\oplus D_4(-1)\oplus \langle -2d\rangle\oplus \langle -2\rangle^{\oplus 5}$ if and only if $T_Y\simeq U^{\oplus 2}\oplus \langle -d\rangle\oplus N\oplus \langle -4\rangle$.\\
\end{thm}

In Section \ref{subsec: subfamilies of Nikuline fourfolds and a conjecture} we study the K3 surface $W$ in the fixed locus of the involution $\sigma$ on the $K3^{[2]}$-type fourfold $X$: we show that there is an isometry between $T_{W}\otimes \Q$ and $T_Y\otimes \Q$, where $T_{\bullet}\otimes \Q$ is transcendental lattice with rational coefficients and $Y$ is the Nikulin orbifold as above (see Proposition \ref{prop: conj over Q}). In particular the Picard number of $W$ is at least 8. Moreover, we conjecture that this isometry holds also with integer coefficients (Conjecture \ref{conj}). We prove the conjecture for many subfamilies of codimension 1, corresponding to Hilbert scheme of points on K3 surfaces with natural symplectic involutions, and for two complete families, see Proposition \ref{prop: examples conjecture d=1 j2} and \ref{prop:conj-Fano-case}. 

In Section \ref{orbRR} we find the Riemann--Roch formula by following step by step the quotient construction of Nikulin fourfolds.
Using the results from \cite{BRZ,Blache,CGM}  for $2$-factorial orbifolds we prove in Corollary \ref{cor: RR for Cartier} and in Proposition 
\ref{prop:RR for Weil divisors v.g. case} the following result.
\begin{thm}
	Let $Y$ be an orbifold of Nikulin type such that $\Pic(Y)=\mathbb{Z}L$, and let $D=\frac{m}{2}L$ be a $\mathbb{Q}$-Cartier divisor on $Y$, $m\in\mathbb{Z}$ odd and $N$ the number of points where $D$ fails to be Cartier. Then \[\chi(\mathcal{O}(D))=\frac{3}{8}\left(\frac{m^4}{24}q(L)^2+m^2q(L)+8\right)-\frac{N}{16}.\]
	
	Moreover, on any orbifold of Nikulin type $Y$ and for any $D\in\Pic(Y)$,
\[\chi(\mathcal{O}(D))=\frac{1}{4}\left(q(D)^2+6q(D)+12\right).\]
\end{thm}

By applying the previous result to some specific divisors on $Y$, we obtain the dimension of certain projective spaces where the quotient $X/\sigma$ or its partial resolution $Y$ are naturally embedded, see Theorem \ref{thm: D1,D2, j1}, \ref{thm: D1,D2, j2,j3} and \ref{thm: D1,D2, jtilde} and Table \ref{table: proj models Y}. This is a first step towards a systematic approach to the geometric descriptions of projective models of Nikulin orbifolds and of their deformations. A specific case is considered in Section \ref{Ndeg2}. 

Indeed, in Section \ref{Ndeg2} we study orbifolds of Nikulin type of degree $2$.
Note that there are two possible values of the divisibility of the polarization, $1$ or $2$; we consider the case of divisibility $1$.
Special examples of such orbifolds are constructed as  quotients of  fourfolds $X$ of $K3^{[2]}$-type with Picard lattice being an extension of index two of 
$\langle 4\rangle\oplus E_8(-2)$ by a symplectic involution. The polarization of degree $4$  on $X$ which is orthogonal to the $E_8(-2)$ summand gives a $2:1$ map \cite{IKKR} to an EPW quartic in the cone $C(\PP^2\times \PP^2)\subset \PP^9$. After projecting from the invariant $\PP^2$ we obtain a complete intersection $(3,4)$ in $\PP^6$ that is singular in codimension $2$ along a surface of degree $52$. From the results in Section \ref{sec: 4-fold with involution} we deduce that the image of the projection is the projective model of the quotient of $X$ by the symplectic involution. By deforming this example and knowing part of the monodromy of orbifolds of Nikulin type (see \cite{MR}) we prove that a general example has a similar description 

\begin{corr}The general element $Y$ in a family of  orbifolds of Nikulin type with a polarization of degree $2$ and divisibility $1$ is a double cover of a special complete intersection $(3,4)$ in $\PP^6$ branched along a surface of degree $48$.
\end{corr}

\section*{Acknowledgements}

The authors want to thank Arvid Perego for an interesting conversation on symplectic orbifolds and Gregoire Menet and Ulrike Riess for many precious discussions and for sharing a preliminary draft of their forthcoming paper.

G.K. is supported by the project Narodowe Centrum Nauki 2018/30/E/ST1/00530. M.K is supported by the project Narodowe Centrum Nauki 2018/31/B/ST1/02857.  

\section{Fourfolds of $K3^{[2]}$-type with a symplectic involution}\label{sec: 4-fold with involution}
We are interested in fourfolds of $K3^{[2]}$-type admitting a symplectic involution and mainly in the projective ones. We will describe the general member of families satisfying these conditions first in a lattice theoretic way and then giving a model as (twisted) moduli space of sheaves on K3 surfaces.
From now on let $X$ be a fourfold of  $K3^{[2]}$-type and $\sigma$ be a symplectic involution on $X$.

\subsection{Lattice theoretic description of $X$}\label{subsec: 4folds with inv. lattices}
Let us fix some notation and recall preliminary results on lattices:
\begin{itemize}
\item The lattice $U$ is the unique even unimodular lattice of rank 2 and signature $(1,1)$; we will denote by $\{u_1,u_2\}$ a basis such that $u_1^2=u_2^2=0$ and $u_1u_2=1$.\item The lattice $E_8$ is the unique even unimodular positive definite lattice of rank 8. 
	\item Given a lattice $M$ and an integer $n\in\Z$, $M(n)$ is the lattice obtained multiplying the bilinear form of $M$ by $n$;
	\item We denote by  $\{b_1,\ldots b_8\}$ the basis of $E_8(-2)$ such that $b_i^2=-4$, $b_ib_{i+1}=2$, $b_3b_8=2$ and the other intersections are zeros.
	\item The lattice $N$, called Nikulin lattice, is an even negative definite rank 8 lattice. It is generated by the classes $r_i$, $i=1,\ldots, 8$ such that $r_i^2=-2$, $r_ir_j=0$ and by the class $\left(\sum_{i=1}^8 r_i\right)/2$. 
\item For $n\in\Z$, $u(n)$ is the discriminant form of $U(n)$; for each $m\in \Z$ and $\alpha\in \Q$ $\mathbb{Z}_m(\alpha)$ is the discriminant cyclic group $\mathbb{Z}_m$ endowed with the quadratic form taking value $\alpha$ on a generator. For short, the discriminant quadratic form of $\mathbb{Z}_m\left(\pm\frac{1}{m}\right)$ is denoted by $(\pm\frac{1}{m})$.
\item The discriminant form of $N$ is $u(2)^{\oplus 3}$ and the discriminant form of $E_8(-2)$ is $u(2)^{\oplus 4}$ (see. \cite[p. 1414]{Nikulin-Factors}).
\item The lattice $L_{K3}$ is the unique even unimodular lattice of rank 22 and signature $(3,19)$ and is isometric to $U^{\oplus 3}\oplus E_8(-1)^{\oplus 2}$.
\item The lattice $L:=L_{K3}\oplus \langle -2\rangle$ and its discriminant form is $(-1/2)$.
\item For every positive integer $d$, the lattice $\Lambda_{2d}$ is isometric to $\langle 2d\rangle\oplus E_8(-2)$. We denote by $h$ a generator of the summand $\langle 2d\rangle$.
\item For even positive integer $d$ the lattice $\widetilde{\Lambda}_{2d}$ is the unique overlattice  of index 2 of $\Lambda_{2d}$ in which $\langle 2d\rangle$ and $E_8(-2)$ are primitively embedded.  
\item The lattice $K_d$ is the negative definite lattice with the following quadratic form
\[
\left[\begin{array}{cc}
-\frac{d+1}{2}&1\\
1&-2
\end{array}\right].
\]
\item The lattice $H_d$ is the indefinite lattice with the following quadratic form
\[ \left[\begin{array}{cc}\frac{d-1}{2}&1\\1&-2\end{array}\right].\]
\end{itemize}

\begin{prop}
The fourfold $X$ admits a symplectic involution if and only if $E_8(-2)$ is primitively embedded in $\Pic(X)$. 

If, moreover, $X$ is projective then there exists a positive $d\in\mathbb{N}$
such that $\Lambda_{2d}\subset \Pic(X)$. 

The Picard group of the very general element in the family of  $\langle 2d\rangle$-polarized $K3^{[2]}$-type fourfold admitting a symplectic involution 
is an overlattice of finite index (possibly equal 1) of 
$\Lambda_{2d}$, with the property that $E_8(-2)$ is primitively
embedded in it.
\end{prop} 
\begin{proof} The first statement is proved by Mongardi in \cite{Mo}. If $X$ is projective, than it admits an ample divisor, which has necessarily a positive self intersection w.r.t. the Beauville--Bogomolov--Fujiki quadratic form. Since $E_8(-2)$ is negative definite, it follows that $\Lambda_{2d}\subset \Pic(X)$ and that if the Picard number of $X$ is the minimal possible, i.e. 9, then $\Pic(X)$ is an overlattice of finite index (possibly equal 1) of 
	$\Lambda_{2d}$, with the property that $E_8(-2)$ is primitively
	embedded in it.
 \end{proof}

\begin{lem}\cite[Proposition 2.2]{vGS}
The overlattices of $\Lambda_{2d}$ containing primitively both $\langle 2d\rangle$ and
$E_8(-2)$ are:
\begin{enumerate}
 \item if $d\equiv 1\mod  2$ only $\Lambda_{2d}$ itself ;
 \item if $d\equiv 0\mod  2$ either $\Lambda_{2d}$ or the unique overlattice $\widetilde{\Lambda}_{2d}$ of index 2 of $\Lambda_{2d}$ in which $\langle 2d\rangle$ and $E_8(-2)$ are primitively embedded. 
\end{enumerate}
The discriminant forms of the lattices  $\Lambda_{2d}$ and $\widetilde{\Lambda}_{2d}$ are $\left(\frac{1}{2d}\right)\oplus u(2)^4$ and $\left(\frac{1}{2d}\right)\oplus u(2)^3$.\\
The lattices $\Lambda_{2d}$ and $\widetilde{\Lambda}_{2d}$ admit a
unique embedding in $L_{K3}$ (up to isometry).
\end{lem}
\begin{proof}
 We only have to compute the discriminant quadratic form of $\widetilde{\Lambda}_{2d}$; this is done in \cite[Corollary 3.7]{CG}. We briefly sketch the proof for the convenience of the reader.

 Let $h, u_{i,j}$ for $i=1,\ldots,4,\ j=1,2$ be a basis of $A_{\Lambda_{2d}}$ on which the discriminant form takes the form $(\frac{1}{2d})\oplus u(2)^{4}$.
 It follows from \cite[Proposition 2.2]{vGS} that we can choose $I$ to be the order two subgroup generated by $dh+v$ where $v\in E_8(-2)$ is such that $v^2=0$ or $1$ respectively when $d\equiv 0\mod 4$ or $d\equiv 2\mod 4$. We suppose that $d\equiv 0\mod 4$ and we can assume that $v=u_{1,1}$. An easy computation shows that $I^\perp=\langle h+u_{1,2},u_{1,1},u_{i,j}|i=2,3,4, j=1,2\rangle$ and thus $\widetilde{\Lambda}_{2d}$ has discriminant quadratic form $(\frac{1}{2d})\oplus u(2)^{3}$. The case $d\equiv 2\mod 4$ is completely analogous.

\end{proof}

In \cite{vGS} an explicit basis for the lattice $\widetilde{\Lambda}_{2d}$ is given:\begin{itemize}
	\item if $d\equiv 2\mod 4$, the lattice $\widetilde{\Lambda}_{2d}$ is generated by the generators of $\Lambda_{2d}$ and  by the class $(h+b_1)/2$;
	\item if $d\equiv 0\mod 4$, the lattice $\widetilde{\Lambda}_{2d}$ is generated by the generators of $\Lambda_{2d}$ and by the class $(h+b_1+b_3)/2$
\end{itemize}
\begin{corollary}
	Let $X$ be a very general element in a family of (possibly not projective) fourfolds of $K3^{[2]}$-type admitting a symplectic involution $\sigma$, then $\NS(X)\simeq E_8(-2)$ and vice versa if $X$ is a fourfold of $K3^{[2]}$-type such that $\NS(X)\simeq E_8(-2)$, then $X$ is non projective and it admits a symplectic involution.
	
Let $X$ be a very general element in a family of projective fourfolds of $K3^{[2]}$-type admitting a symplectic involution $\sigma$. Then either $\Pic(X)\simeq \Lambda_{2d}$ for a certain integer $d>0$ or $\Pic(X)\simeq \widetilde{\Lambda}_{2d}$ for a certain even integer $d>0$.

Vice versa if $X$ is a fourfolds of $K3^{[2]}$-type such that $\Pic(X)$ is isometric either to $\Lambda_{2d}$ for an integer $d>0$ or to $\widetilde{\Lambda}_{2d}$ for an even integer $d>0$, then $X$ is projective and admits a symplectic involution.   
\end{corollary}

We observe that $E_8(-2)$ admits a unique primitive embedding in $L$, whose orthogonal is $U^{\oplus 3}\oplus E_8(-2)\oplus \langle -2\rangle$.

In order to determine the families of projective fourfolds of $K3^{[2]}$-type admitting a symplectic involution, we consider all possible primitive embeddings of the lattices $\Lambda_{2d}$ and $\widetilde{\Lambda_{2d}}$ into $L$.

\begin{prop}\label{prop:embeddings of Lamba(2d)}
For any integer $d>0$ $\Lambda_{2d}$ admits, up to isometry of $L$, the following
primitive embeddings into $L$:
 \begin{enumerate}
 \item  $j_1$  such that $j_1(\Lambda_{2d})^\perp \simeq T_{2d,1}:=U^{\oplus 2}\oplus E_8(-2)\oplus\langle -2d\rangle\oplus\langle -2\rangle$;
 \item  if $d\equiv 1\mod 2$,  $j_2$ such that $j_2(\Lambda_{2d})^\perp \simeq T_{2d,2}:=U^{\oplus 2}\oplus D_4(-1)\oplus\langle -2d\rangle\oplus\langle -2\rangle^{\oplus 5}$;
 \item if $d\equiv 3\mod 4$, $j_3$ such that $j_3(\Lambda_{2d})^\perp \simeq T_{2d,3}:=U^{\oplus 2}\oplus E_8(-2)\oplus K_d$.
\end{enumerate}
 For any $d\equiv 0\mod 2$,
 $\widetilde{\Lambda}_{2d}$ admits a unique primitive embedding $\widetilde{j}$ into $L$, with orthogonal isometric to $\widetilde{T}_{2d}:=U^{\oplus 2}\oplus D_4(-1)\oplus\langle -2d\rangle\oplus\langle -2\rangle^{\oplus
 5}$.
\end{prop}
\begin{proof}
We study first possible primitive embeddings of $\Lambda_{2d}$ inside $L$. The first embedding $j_1$ is simply obtained by composing
the embedding of $\Lambda_{2d}$ inside $L_{K3}$ with the embedding
of this one inside $L$. This is unique up to isometry if
$d\equiv 0\mod 2$.

When $d\equiv 1\mod 2$, an application of \cite[Proposition 1.15.1]{Nikulin-Forms}
shows that there is a second possibility: indeed, in this case
$A_{\Lambda_{2d}}$ contains a subgroup $H$ of order two to which
the discriminant form restricts as $(-\frac{1}{2})$.
Standard computations in this case produce the embedding $j_2$ if
$d\equiv 1\mod 4$, and the embeddings $j_2$ and $j_3$ if $d\equiv
3\mod 4$. Up to isometry these are the only possibilities.

Concerning the primitive embeddings of $\widetilde{\Lambda}_{2d}$, $\widetilde{j}$ is again obtained by
composing the embedding of $\widetilde{\Lambda}_{2d}$ inside
$L_{K3}$ with the embedding of this one inside $L$. The fact that
it is the only possible one comes by an application of
\cite[Proposition 1.15.1]{Nikulin-Forms}: we have
$A_L\cong \Z_2(-\frac{1}{2})$, whereas the quadratic form on
$A_{\widetilde{\Lambda}_{2d}}$ takes values in $\Z/2\Z$ on any
subgroup of order two; as a consequence, the only possible choice
for two isometric subgroups inside $A_L$ and
$A_{\widetilde{\Lambda}_{2d}}$ is $H=\lbrace
0\rbrace$, and the discriminant form of the
orthogonal is exactly $(-q_{\widetilde{\Lambda}_{2d}})\oplus
q_{A_L}=u(2)^{\oplus 3}\oplus(-\frac{1}{2d})\oplus (-\frac{1}{2})$. From
\cite[Proposition 1.8.2]{Nikulin-Forms}, we have $u(2)^{\oplus 3}\oplus(-\frac{1}{2d})\oplus
(-\frac{1}{2})\simeq (\frac{1}{2})^{\oplus
3}\oplus(-\frac{1}{2})^{\oplus 4}\oplus  (-\frac{1}{2d})$.
Moreover, it is easy to show that $(\frac{1}{2})^{\oplus
3}\oplus(-\frac{1}{2})^{\oplus 4}\simeq v(2)\oplus
(-\frac{1}{2})^{\oplus 5}$, and this yields
$\widetilde{j}(\widetilde{\Lambda}_{2d})^\perp\cong
\widetilde{T}_{2d}$, because such an orthogonal is unique in its
genus and the map $O(\widetilde{T}_{2d})\rightarrow
O(q_{A_{\widetilde{T}_{2d}}})$ is surjective.
\end{proof}

To recap, if $X$ is a very general projective fourfold of $K3^{[2]}$-type admitting a symplectic involution, there are the following possibilities for $\Pic(X)$ and $T_X$
{\small{
\begin{align}\label{table: j,NS,T}
	\begin{array}{|c|c|c|c|c|}
		\hline
		\mbox{ Condition on d}&\mbox{Embed. }\Pic(X)\subset L&\Pic(X)&T_X\\
		\hline
		\forall d\in\mathbb{N}&j_1&\Lambda_{2d}&T_{2d,1}:=U^{\oplus 2}\oplus E_8(-2)\oplus \langle -2d\rangle\oplus \langle -2\rangle\\
		\hline
		d\equiv 1\mod 2&j_2&\Lambda_{2d}&T_{2d,2}:=U^{\oplus 2}\oplus D_4(-1)\oplus \langle -2d\rangle\oplus \langle -2\rangle^{\oplus 5}\\
		\hline
		d\equiv 3\mod 4&j_3&\Lambda_{2d}&T_{2d,3}:=U^{\oplus 2}\oplus E_8(-2)\oplus K_d\\
		\hline
		d\equiv 0\mod 2&\widetilde{j}&\widetilde{\Lambda}_{2d}&\widetilde{T}_{2d}:=U^{\oplus 2}\oplus D_4(-1)\oplus \langle -2d\rangle\oplus \langle -2\rangle^{\oplus 5}\\
		\hline
\end{array}\end{align}
}}
\\
As observed before, if $X$ is a very general non-projective fourfold of $K3^{[2]}$-type admitting a symplectic involution, then  $\Pic(X)=E_8(-2)$ and $T_X=U^{\oplus 3}\oplus E_8(-2)\oplus \langle -2\rangle$.\\

The embeddings $j_a$, $a=1,2,3$ and $\widetilde{j}$ in $L$ can be explicitly determined (and this will be used in the following).

Let $X$ be a $K3^{[2]}$-type fourfold admitting a symplectic involution $\iota$. Fix a basis of $H^2(X,\Z)\simeq U^{\oplus 3}\oplus E_8(-1)\oplus E_8(-1)\oplus \langle -2\rangle$: there exists an isometry between $H^2(X,\Z)$ and $L=U^{\oplus 3}\oplus E_8(-1)\oplus E_8(-1)\oplus \langle -2\rangle$ such that the involution $\iota^*\in O(H^2(X,\Z))$ switches the two copies of $E_8(-1)$ and acts as the identity on $U\oplus U\oplus U\oplus\langle -2\rangle$. 
We denote by $e_i$, (resp. $f_i$), $i=1,\ldots, 8$ a basis of the first (resp. second) copy of $E_8(-1)$ in $E_8(-1)\oplus E_8(-1)$, and by $b_i$ a basis of $E_8(-2)$. We  fix two different embeddings of the lattice $E_8(-2)$ in $E_8(-1)\oplus E_8(-1)$:
$$\begin{array}{llll}\lambda_+(b_i)&=&e_i+f_i& i=1,\ldots ,8\\
	\lambda_-(b_i)&=&e_i-f_i& i=1,\ldots ,8.\\
\end{array}$$

In particular $H^2(X,\Z)^{\iota^*}=U^{\oplus 3}\oplus\lambda_+(E_8(-2))\oplus\langle -2\rangle\simeq U^{\oplus 3}\oplus E_8(-2)\oplus\langle -2\rangle$ and  $(H^2(X,\Z)^{\iota^*})^{\perp}=\lambda_-(E_8(-2))\simeq E_8(-2)$.

Let $h\in H^2(X,\Z)$ be a $\iota$-invariant primitive class with self-intersection  $2d>0$. Let us denote by $j(h)$ an embedding of $h$ in $H^2(X,\Z)\simeq L$. Since the polarization $h$ is invariant for $\iota$, $j(h)\in H^2(X,\Z)^{\iota^*}\simeq U^{\oplus 3}\oplus\lambda_{+}(E_8(-2))\oplus\langle -2\rangle$ and thus it corresponds to a vector of the form $(\underline{u},\underline{w},\underline{v},\underline{x},\underline{y},k)$ such that $\underline{x}=\underline{y}$. We now consider different embeddings $j$, which produce different families of fourfolds of $K3^{[2]}$-type with a symplectic involution. 

\begin{prop}\label{prop: j1 embedding}
	Let $d$ be a positive integer and let $$j_1(h):=\left(\left(\begin{array}{c}1\\d\end{array}\right),\left(\begin{array}{c}0\\0\end{array}\right), \left(\begin{array}{c}0\\0\end{array}\right), \underline{0}, \underline{0},0\right).$$
	The embedding $(j_1,\lambda_{-}):\langle 2d\rangle\oplus E_8(-2)\rightarrow L$ is a primitive embedding and there exist fourfolds of $K3^{[2]}$-type $X_1$ such that $\NS(X_1)\simeq (j_1,\lambda_{-})\left(\langle 2d\rangle\oplus E_8(-2)\right)\simeq \Lambda_{2d}$ and $T_{X_1}\simeq T_{2d,1} $.
\end{prop}
\proof The embedding $(j_1,\lambda_{-})$ is clearly  primitive, hence there exist fourfolds of $K3^{[2]}$-type admitting this lattice as N\'eron--Severi group. Since $j_1$ restricts to an embedding of $h$ in $U$ and $\lambda_{-}$ restricts to an embedding of $E_8(-2)$ in $E_8(-1)\oplus E_8(-1)$, one can compute separately the orthogonal in the different direct summands, finding that the orthogonal to $\NS(X_1)$ is $\langle -2d\rangle\oplus U\oplus U\oplus\lambda_{+}(E_8(-2))\oplus\langle -2\rangle\simeq T_{2d,1}$.
\endproof

\begin{prop}\label{prop: j2 embedding}
	Let $d$ be an odd positive integer. Let $$\begin{array}{ll}j_2(h):=\left(\left(\begin{array}{c}2\\2k+2\end{array}\right),\left(\begin{array}{c}0\\0\end{array}\right), \left(\begin{array}{c}0\\0\end{array}\right), \underline{e_1}, \underline{e_1},1\right)&\mbox{ if }d=4k+1\\
		j_2(h):=\left(\left(\begin{array}{c}2\\2k+2\end{array}\right),\left(\begin{array}{c}0\\0\end{array}\right), \left(\begin{array}{c}0\\0\end{array}\right), \underline{e_1}+\underline{e_3}, \underline{e_1}+\underline{e_3},1\right)&\mbox{ if }d=4k-1.\end{array}$$
	The embedding $(j_2,\lambda_{-}):\langle 2d\rangle\oplus E_8(-2)\rightarrow L$ is a primitive embedding and there exist fourfolds of $K3^{[2]}$-type $X_2$ such that $\NS(X_2)\simeq (j_2,\lambda_{-})\left(\langle 2d\rangle\oplus E_8(-2)\right)\simeq \Lambda_{2d}$ and $T_{X_2}\simeq T_{2d,2}$.
\end{prop}
\proof The embedding $(j_2,\lambda_{-})$ is clearly primitive, hence there exist fourfolds of $K3^{[2]}$-type admitting this lattice as N\'eron--Severi group. By Proposition \ref{prop:embeddings of Lamba(2d)} there is an embedding of $\langle 2d\rangle\oplus E_8(-2)$ in $U^{\oplus 3}\oplus E_8(-1)^{\oplus 2}\oplus\langle -2\rangle$ which is not equivalent to $j_1$, computed in Proposition \ref{prop: j1 embedding}. 

Let $\underline{x}=\left\{\begin{array}{ll}\underline{e_1} &\mbox{if }d\equiv 1\mod 4\\\underline{e_1}+\underline{e_3}&\mbox{if }d\equiv 3\mod 4.\end{array}\right..$ By direct computation, the orthogonal lattice  $((j_2,\lambda_-)(\Lambda_{2d}))^\perp$ is spanned by the following vectors:

\begin{displaymath}
	\left(\left(\begin{array}{ll}
		-1\\k+1
	\end{array}\right),\underline{0},\underline{0},\underline{0},\underline{0},0\right),\ 
	\left(\left(\begin{array}{ll}
		0\\1
	\end{array}\right),\underline{0},\underline{0},\underline{0},\underline{0},1\right),\ 
	\left(\underline{0},\underline{0},\underline{0},\underline{w},\underline{w},0\right), \underline{w}\in (\underline{x})^\perp_{E_8(-1)},
\end{displaymath}
\begin{displaymath}
	b:=\left(\underline{0},\underline{0},\underline{0},\underline{y},\underline{y},1\right)\ \mathrm{with}\ 
\underline{y}=\left\{\begin{array}{ll}\underline{e_2} &\mbox{if }d\equiv 1\mod 4\\\underline{e_4}&\mbox{if }d\equiv 3\mod 4.\end{array}\right.
	\end{displaymath}
	
One can directly check that this lattice is isometric to $T_{2d,2}$.
\endproof

\begin{prop}\label{prop: j3 embedding}
	Let $d$ be a positive integer such that $d\equiv 3\mod 4$. Let $$j_3(h):=\left(\left(\begin{array}{c}2\\(d+1)/2\end{array}\right),\left(\begin{array}{c}0\\0\end{array}\right), \left(\begin{array}{c}0\\0\end{array}\right), \underline{0}, \underline{0},1\right).$$
	
	The embedding $(j_3,\lambda_{-}):\langle 2d\rangle\oplus E_8(-2)\rightarrow L$ is a primitive embedding and there exist fourfolds of $K3^{[2]}$-type  $X_3$ such that $\NS(X_3)\simeq (j_3,\lambda_{-})\left(\langle 2d\rangle\oplus E_8(-2)\right)\simeq\Lambda_{2d}$ and $T_{X_3}\simeq T_{2d,3}$.
\end{prop}
\proof The embedding $(j_3,\lambda_{-})$ is clearly a primitive embedding of $\langle 2d\rangle\oplus E_8(-2)$ in $L$ and hence there exist fourfolds of $K3^{[2]}$-type admitting this lattice as N\'eron--Severi group. 
Since $j_3$ restricts to an embedding of $h$ in $U\oplus \langle -2\rangle$, one can compute the orthogonal of $j_3(h)$ in $U\oplus\langle -2\rangle$, which is generated by $\left(\left(\begin{array}{c}0\\1\end{array}\right),1\right)$ and 
$\left(\left(\begin{array}{c}1\\-(d+1)/4\end{array}\right),0\right)$, with intersection form equal to $K_d$, so that $T_{X_3}\simeq T_{2d,3}$.
\endproof

\begin{prop}\label{prop: jtilde embedding}
	Let $d$ be an even positive integer. Let $$\begin{array}{ll}\widetilde{j}(h):=\left(\left(\begin{array}{c}2\\2k\end{array}\right),\left(\begin{array}{c}0\\0\end{array}\right), \left(\begin{array}{c}0\\0\end{array}\right), \underline{e_1}, \underline{e_1},0\right)&\mbox{ if }d=4k-2\mbox{ and }\\
		\widetilde{j}(h):=\left(\left(\begin{array}{c}2\\2k\end{array}\right),\left(\begin{array}{c}0\\0\end{array}\right), \left(\begin{array}{c}0\\0\end{array}\right), \underline{e_1}+\underline{e_3}, \underline{e_1}+\underline{e_3},0\right)&\mbox{ if }d=4k-4.
	\end{array}$$
	The embedding $(\widetilde{j},\lambda_{-}):\langle 2d\rangle\oplus E_8(-2)\rightarrow L$ is not a primitive embedding and the primitive closure of $(\widetilde{j},\lambda_{-})\left(\langle 2d\rangle\oplus E_8(-2)\right)$ is isometric to $\widetilde{\Lambda}_{2d}$. There exist fourfolds of $K3^{[2]}$-type $\widetilde{X}$ such that $\NS(\widetilde{X})\simeq \widetilde{\Lambda}_{2d}$ and  $T_{\widetilde{X}}\simeq \widetilde{T}_{2d}$.\end{prop}

\proof Let us consider the case $d=4k-2$, i.e. $d\equiv 2\mod 4$. The embedding $(\widetilde{j},\lambda_{-})$ is not primitive, since the class $\widetilde{j}(h)+\lambda_{-}(b_1)$ is two-divisible in $U\oplus U\oplus U\oplus E_8(-1)\oplus E_8(-1)\oplus \langle-2\rangle$, whereas $h+b_1$ is primitive inside $\langle 2d\rangle\oplus E_8(-2)$. By adding the class $\left(\widetilde{j}(h)+\lambda_{-}(b_1)\right)/2$ to $(\widetilde{j},\lambda_{-})(\langle 2d\rangle\oplus E_8(-2))$ one obtains a primitive embedding of $\widetilde{\Lambda}_{2d}$ in $L$. In particular there exists a fourfold of $K3^{[2]}$-type with $\NS(\widetilde{X})\simeq \widetilde{\Lambda}_{2d}$ and, by computing its orthogonal complement, one finds $T_{\widetilde{Y}}\simeq U^{\oplus 2}\oplus \langle -2d\rangle\oplus D_4(-1)\oplus \langle -2\rangle^{\oplus 5}$.

The case $d=4k-4$ is analogous.\endproof

\subsection{Models of $X$ as moduli space of sheaves on a K3 surface}\label{subsec: 4folds inv. models}

In this section we provide at least one model of the very general $X$ in one of the families described in Table \eqref{table: j,NS,T}, i.e. of the very general member of each family of projective fourfolds of $K3^{[2]}$-type admitting a symplectic involution. Each of them will be described either as Hilbert scheme of a certain K3 surface or as a, possibly twisted, moduli space of sheaves on a K3 surface. Hence the main results of this section are summarized in Table \eqref{table: proj models X}.

One needs two preliminary definitions in order to list all cases.
\begin{defi}
	If $d\equiv 3\mod 4$, we denote by $\left(\langle
	2d\rangle\oplus\langle -2\rangle ^{\oplus 7}\right)'$ the overlattice of
	$\langle 2d\rangle\oplus\langle -2\rangle ^{\oplus 7}=\Z t\bigoplus\oplus_i
	\Z n_i$ obtained by adding to $\langle 2d\rangle\oplus\langle
	-2\rangle ^{\oplus 7}$ the class $\left(t+\sum_i n_i\right)/2$.
	\end{defi}

\begin{lem} The lattice $\langle
	2d\rangle\oplus\langle -2\rangle ^{\oplus 7}$ admits a unique primitive embedding in $L_{K3}$ and its orthogonal is uniquely determined and isometric to $U^{\oplus 2}\oplus D_4(-1)\oplus \langle -2d\rangle\oplus \langle -2\rangle^{\oplus 5}$.
	
	If $d\equiv 3\mod 4$ the lattice $\left(\langle
	2d\rangle\oplus\langle -2\rangle ^{\oplus 7}\right)'$ admits a unique primitive embedding in $L_{K3}$ and its orthogonal is uniquely determined and isometric to $U\oplus U\oplus N\oplus K_d$.
\end{lem}
\begin{proof}
The discriminant quadratic form of $Q:=\langle 2d\rangle\oplus \langle -2 \rangle^{\oplus 7}$ is $\left(\frac{1}{2d}\right)\oplus\left(-\frac{1}{2}\right)^{\oplus 7}$. Since $L_{K3}$ is unimodular, the orthogonal $Q^\perp$ needs to have discriminant quadratic form $\left(-\frac{1}{2d}\right)\oplus\left(\frac{1}{2}\right)^{\oplus 7}\simeq \left(-\frac{1}{2d}\right)\oplus v(2)\oplus\left(-\frac{1}{2}\right)^{\oplus 5}$ and signature $(2,12)$: there is, up to isometry, a unique lattice with these properties, which is $U^{\oplus 2}\oplus D_4(-1)\oplus \langle -2d\rangle\oplus \langle -2\rangle^{\oplus 5}$, thus the embedding is unique up to isometry of $L_{K3}$.

The discriminant quadratic form of $Q':=(\langle 2d\rangle\oplus \langle -2 \rangle^{\oplus 7})'$ is $\left(\frac{2}{d}\right)\oplus u(2)^{\oplus 3}$, hence its orthogonal in $L_{K3}$ has discriminant quadratic form $\left(-\frac{2}{d}\right)\oplus u(2)^{\oplus 3}$ and signature $(2,12)$: again, there is, up to isometry, a unique lattice with these properties, which is $U^{\oplus 2}\oplus N \oplus K_d$.
\end{proof}

The previous lemma implies that there exists a well defined family of K3 surfaces which is polarized with the lattice $\langle
2d\rangle\oplus\langle -2\rangle ^{\oplus 7}$ (resp. $\left(\langle
2d\rangle\oplus\langle -2\rangle ^{\oplus 7}\right)'$).
\begin{defi}
For any positive integer $d$, $S_d$ is a K3 surface such that $\NS(S_d)=\langle
2d\rangle\oplus\langle -2\rangle ^{\oplus 7}$.

For the positive integer $d$ such that $d\equiv 3 \mod 4$, $Z_d$ is a K3 surface such that $\NS(Z_d)=\left(\langle
2d\rangle\oplus\langle -2\rangle ^{\oplus 7}\right)'$\end{defi}	

By the previous lemma, the transcendental lattices of the surfaces $S_d$ and $Z_d$ are respectively $T_{S_d}\simeq U^{\oplus 2}\oplus D_4(-1)\oplus \langle -2d\rangle\oplus \langle -2\rangle^{\oplus 5}$ and $T_{Z_d}\simeq U\oplus U\oplus N\oplus K_d$.\\

In the following we will denote by $H'$ a primitive vector in $\langle
2d\rangle\oplus\langle -2\rangle ^{\oplus 7}$ or $\left(\langle
2d\rangle\oplus\langle -2\rangle ^{\oplus 7}\right)'$ whose square is $2$. It surely exists by Lagrange's four squares theorem.\\

Now we can write the table where we summarize the birational models given for $X$: in the first column we identify the family of fourfolds we are considering (and this is done by exhibiting the embedding $\Pic(X)\subset L$, using the results in \eqref{table: j,NS,T}); in the second column we declare which K3 surface is associated to the model; in the third we describe the model; if the model of $X$ is as moduli space of sheaves determined by a Mukai vector, in the fourth column we write the Mukai vector (we omit the element in the Brauer group giving the twist, when needed); in the last column we give the reference to the proposition where we describe the model and prove that it is the required one.

\begin{align}\label{table: proj models X}
	\begin{array}{|c|c|c|c|c|}
		\hline
		\mbox{Embedding }\Pic(X)\subset L&K3&\mbox{ model }&\mbox{ v }&\mbox{Prop.}\\
		\hline
		j_1,\ d\equiv 1\mod 2&S_d&M_v(S_d,\beta)&(0,H',2)&\mbox{\ref{prop: d odd twisted moduli}}\\
		\hline
		j_1,\ d\equiv 0\mod 2&S_d&M_v(S_d,\beta)&(4,\sum_{i=1}^7n_i,2)&\mbox{\ref{prop: model d even}}\\
		\hline
		j_2,\ d\equiv 1\mod 2&S_d&S_d^{[2]}&-&\mbox{\ref{prop: d odd j2}}\\
		\hline
		j_3, \ d\equiv 3\mod 4&Z_d&M_v(Z_d,\beta)&(0,H',2)&\mbox{\ref{prop: third case twisted moduli}}\\
		\hline
		\widetilde{j},\ d\equiv 0\mod 2&S_d&M_v(S_d)&(2,\sum_{i=1}^7n_i,4)&\mbox{\ref{prop: model d even}}\\
		\hline
\end{array}\end{align}

The easiest description of the fourfold that we obtain is the one associated to the embedding $j_2:\Pic(X)\hookrightarrow L$, indeed in this case $X$ is (birational to) a Hilbert scheme of points on a K3 surface, by the following. 
\begin{prop}\label{prop: d odd j2}
Let $d$ be an odd number. Then $S_d^{[2]}$ is a $(\Lambda_{2d},j_2)$-polarized fourfold.
	\end{prop}
\begin{proof}
	The transcendental lattice of $S_d$ is $T_{S_d^{[2]}}\simeq U^{\oplus 2}\oplus D_4(-1)\oplus \langle -2d\rangle\oplus \langle -2\rangle^{\oplus 5}$. Hence $T_{S_d^{[2]}}\simeq U^{\oplus 2}\oplus D_4(-1)\oplus \langle -2d\rangle\oplus \langle -2\rangle^{\oplus 5}\simeq T_{2d,2}$.
\end{proof}

\begin{rem}\label{rem:involution-Hilbert-schemes}
Note that there is no natural
symplectic involution on these Hilbert schemes.
It is a nice problem to construct such involutions on birational models of those Hilbert
schemes. 

When $d=1$, this
family is exactly the one given by Hilbert squares of $K3$
surfaces which are $(\langle 2\rangle\oplus \langle
-2\rangle^{\oplus 7})$-polarized. As above, we denote by $S_1$ a generic surface in this space. 

The surface $S_1$ is a double cover of a del Pezzo surface of degree 2, denoted by $dP_2$, thus it admits a non symplectic involution, which is the cover involution. We denote it as $\iota_{S_1}$ and we observe that it acts as the identity on $\NS(S_1)$. Moreover, since the anticanonical model of $dP_2$ exhibits $dP_2$ as double cover of $\mathbb{P}^2$ branched on a quartic curve, the surface $S_1$ admits a model (induced by the anticanonical one of $dP_2$) as a quartic hypersurface in $\mathbb{P}^3$ which does not contain lines. Therefore the fourfold $X\simeq S_1^{[2]}$ admits two non-symplectic involutions: one is $\iota_{S_1}^{[2]}$, the natural involution induced by $\iota_{S_1}$, and the other is Beauville's involution $\beta$ (see \cite[Proposition 11]{Beauville} for the definition). The isometry $(\iota_{S_1}^{[2]})^*$ acts as the identity on $\NS(X)$ and as minus the identity on $T_{X}$, hence it commutes with every isometry induced by an involution on $X$ (since they commute both on the transcendental lattice and on the N\'eron--Severi group). In particular $\iota_{S_1}^{[2]}$ and $\beta$ are two commuting non symplectic involutions, whose composition is necessarily a symplectic involution on $X$.
Such an involution
can be constructed also on a birational model, as done in \cite{MT}.

In the case $d=3$, by Proposition \ref{prop: d odd j2} examples of $(\Lambda_6,j_2)$-polarized fourfolds are given by Hilbert squares of $K3$ surfaces $S_3$ which are $(\langle 6\rangle\oplus\langle -2\rangle^{\oplus 7})$-polarized. In this case, one can show that such Hilbert squares are in fact birational to the Fano varieties of cubic fourfolds with $8$ nodes \cite[Thm.~1.1]{L}. It is an open question whether it is possible to describe geometrically a symplectic involution on these manifolds.
\end{rem}
\begin{prop}\label{prop: d odd twisted moduli}
Let $d$ be an odd number. There exist a Brauer class $\beta\in H^2(\oo^*_{S_d})_2$ and a
Mukai vector $v\in H^*(S_d,\Z)$ such that the moduli space
$X=M_v(S,\beta)$ is a $(\Lambda_{2d},j_1)$-polarized fourfold of
$K3^{[2]}$-type.
\end{prop}
\begin{proof}
The transcendental lattice $T_{S_d}$ of $S_d$ is of the form $U\oplus W$
for $W$ an even hyperbolic lattice; we denote by $f_1,f_2$ a basis
of the hyperbolic plane $U$. Then we consider $B=\frac{f_1}{2}\in
T_S\otimes\Q$ and $\beta\in H^2(\oo^*_{S_d})_2$ the Brauer class of
order two corresponding to $(\_,B):T_S\rightarrow \Z/2\Z$. The
twisted Picard group $\Pic(S,\beta)$ is thus the sublattice of
$H^*(S,\Z)$ generated by $\Pic(S_d)$, $(0,0,1)$ and $(2,f_1,0)$,
hence it is isomorphic to $U(2)\oplus \Pic(S)$, and its orthogonal
in the Mukai lattice is isometric to $U(2)\oplus W$. It follows
from work of Yoshioka \cite[Section 3]{Yoshioka} that $\Pic(M_v(S_d,\beta))\cong v_B^\perp\cap
\Pic(S_d,\beta)$ and that the transcendental lattice of
$M_v(S_d,\beta)$ is isometric to $U(2)\oplus W$.

We conclude by choosing as Mukai vector $v=(0,H',2)$ where $H'\in
\Pic(S)$ is a primitive effective class of square two. The orthogonal $P$ of $H'$ in
$\Pic(S)$ is a negative definite lattice with rank and length 7
and discriminant group $\Z_{2d}\oplus (\Z_2)^{\oplus 6}$ with discriminant quadratic form $q=\left(\frac{1}{2d}\right)\oplus v(2)\oplus \left(-\frac{1}{2}\right)^{\oplus 4}$.  
For such a choice we have $v_B=v$ primitive of square two and the
orthogonal to $v$ in $U(2)\oplus \Pic(S)$ is a hyperbolic lattice
of rank 9 and discriminant group $\Z_{2d}\oplus (\Z_2)^{\oplus
8}$. Its $2$-adic component is isometric to the one of $\langle
2\rangle\oplus\langle -2\rangle^8\simeq  \langle 2\rangle\oplus
E_8(-2)$ and there is only one even indefinite lattice in this
genus by \cite[Theorem 1.13.2]{Nikulin-Forms}. Thus the orthogonal
to $v$ is isometric to $\Lambda_{2d}$.
\end{proof}

\begin{prop}\label{prop: third case twisted moduli}
Let $d$ be a positive integer such that $d\equiv 3\mod 4$. There exist a Brauer class $\beta\in H^2(\oo^*_{Z_d})_2$
and a Mukai vector $v\in H^*(Z_d,\Z)$ such that the moduli space
$X=M_v(Z_d,\beta)$ is a $(\Lambda_{2d},j_3)$-polarized fourfold of
$K3^{[2]}$-type.
\end{prop}
\begin{proof}
	
Denoted by $W$ the lattice $U\oplus N\oplus K_d$, it holds 
$$T_{2d,3}\simeq U^{\oplus 2}\oplus
E_8(-2)\oplus K_d\simeq U(2)\oplus U\oplus N\oplus K_d\simeq U(2)\oplus W$$ and  $T_{Z_d}\simeq U\oplus U\oplus
N\oplus K_d\simeq U\oplus W$.  
Now reasoning as in Proposition \ref{prop: d odd twisted moduli}, one chooses
$B=\frac{f_1}{2}\in T_Z\otimes\Q$ and $\beta\in H^2(\oo^*_{Z_d})_2$
the Brauer class of order two corresponding to
$(\_,B):T_{Z_d}\rightarrow \Z/2\Z$. So $\Pic(M_v(Z_d,\beta))\cong
v_B^\perp\cap \Pic(Z_d,\beta)$ and the transcendental lattice
of $M_v(Z_d,\beta)$ is isometric to $U(2)\oplus W\simeq T_{2d,3}$.

We conclude by choosing as Mukai vector $v=(0,H',2)$ where $H'\in
\Pic(Z_d)\simeq \left(\langle 2d\rangle\oplus \langle -2\rangle^{\oplus 7}\right)'$ is a primitive effective class of square two.
\end{proof}

\begin{prop}\label{prop: model d even}
If $d\equiv 0\mod 2$, then:
\begin{itemize}
\item a general fourfold of $K3^{[2]}$-type
$(\Lambda_{2d},j_1)$-polarized is birational to $M_v(S_d,\beta)$ where $v=(4,\sum_in_i,2)$ and $\beta$ is as above; \item a general fourfold of $K3^{[2]}$-type
$(\widetilde{\Lambda}_{2d},\tilde{j})$-polarized is birational to
$M_{w}(S_d)$ with $w=(2, \sum n_i, 4)\in H^*(S_d,\Z)$.
\end{itemize}
\end{prop}
\begin{proof} Let us fix $\beta$ as in Proposition \ref{prop: d odd twisted moduli}. Then, since $T_{S_d}\simeq U^{\oplus 2}\oplus D_4(-1)\oplus \langle -2\rangle^{\oplus 5}\oplus \langle-2d\rangle$, $T_{M_v(S_d,\beta)}\simeq U\oplus U(2)\oplus D_4(-1)\oplus \langle -2\rangle^{\oplus 5}\oplus \langle-2d\rangle$ for every possible choice of the Mukai vector $v$. Moreover, the twisted Picard group $\Pic(S_d,\beta)$ is $U(2)\oplus \Pic(S_d)$ (as in proof of Proposition \ref{prop: d odd twisted moduli}) and it is generated by $(0,0,1)$, $(2,f_1,0)$, $(0,n_i,0)$, $i=1,\ldots 7$, $(0,t,0)$ (where $t$, $n_i$ are the generators of $\Pic(S_d)$, $t^2=2d$, and $f_1$ is as in Proposition \ref{prop: d odd twisted moduli}). We now fix $v=(4,\sum_in_i,2)$, then $v_B=(4,\sum_in_i+2f_1,2)\in H^*(S_d,\Z)$ and we compute $v_B^\perp \cap \Pic(S_d,\beta)$. It is generated by $(2,f_1,-1)$, $(0,2n_1,1)$, $(0,n_i-n_{i+1},0)$, $i=1,\ldots 6$, $(0,t,0)$. One can directly check that $(0,t,0)$ is orthogonal to all the other generators and the form computed on all the other generators is $R(2)$ where $R$ is an even negative definite unimodular lattice of rank 8. It follows that $R\simeq E_8(-1)$ and so the orthogonal to $v_B$ in $\Pic(S_d,\beta)$ is isometric to $E_8(-2)\oplus \langle 2d\rangle\simeq \Lambda_{2d}$. Hence $M_v(S_d,\beta)$ is
$(\Lambda_{2d},j_1)$-polarized  and gives a birational model of the  general $(\Lambda_{2d},j_1)$-polarized fourfold of $K3^{[2]}$-type.

To prove the similar result for a general fourfold of $K3^{[2]}$-type $(\widetilde{\Lambda}_{2d},\tilde{j})$-polarized we observe that $T_{S_d}\cong \widetilde{T}_{2d}$. Moreover, the $(1,1)$-part in $H^*(S_d,\Z)$ is $U\oplus \Pic(S_d)$. 
Next, we observe that $\widetilde{\Lambda}_{2d}\cong \langle 2d\rangle \oplus N$, where $N$ is the Nikulin lattice, obtained by $\langle -2\rangle^{\oplus 8}$ by gluing the class $n:=\sum r_i/2$ and it is generated by the first seven roots $r_1,\ldots,r_7$ and by $n$ such that $n^2=-4$ and $nr_i=-1$.

Let $f_1,f_2,t,n_1,\ldots,n_7$ be a basis of $U\oplus\Pic(S_d)$.
Consider now the explicit primitive embedding $\langle 2d\rangle \oplus N\subset U\oplus\Pic(S_d)$ which sends the $\langle 2d\rangle$ summand in the lattice spanned by $t$ and which sends $r_i\mapsto n_i+f_1$ for $i=1,\ldots, 7$, $n\mapsto 2f_1-f_2$. The Mukai vector $w=(2,\sum_in_i,4)$ is $4f_1+2f_2+n_1+\ldots+n_7$ and its orthogonal is spanned by $t$, $n$ and $r_i$ with $i=1,\ldots 7$. So the orthogonal to the Mukai vector $w$ in $U\oplus \Pic(S_d)$ is isometric to $\langle 2d\rangle \oplus N\simeq \widetilde{\Lambda}_{2d}$ and this ends the proof.
\end{proof}

\begin{rem}{(Induced automorphisms from autoequivalences.)}
The symplectic automorphism considered in Proposition \ref{prop: model d even}
is induced by a symplectic autoequivalence on $D^{b}(S)$
that are not induced by a symplectic action on $S$.
The result \cite[Prop~1.4]{BO} gives a way to further investigate
these symplectic involutions.
 If \cite[Prop~1.4]{BO} is generalized for twisted sheaves then this would give away a way to study also the other involutions considered here.
\end{rem}

\section{Nikulin orbifolds of dimension four}\label{sec: Nikulin fourfold}

After having described the moduli spaces of projective fourfolds $X$ of $K3^{[2]}$-type admitting a symplectic involution $\iota$, we now turn to the study of the quotients. It is well-known, since work of Fujiki \cite{Fujiki}, that the quotient does not admit a crepant resolution of singularities. Nevertheless, there is a partial resolution $Y\rightarrow X/\iota$ which is a so-called {\it irreducible sympletic orbifold}.
\begin{defi}
The irreducible symplectic orbifolds obtained as partial resolution of $X/\iota$ for a certain manifold of $K3^{[n]}$-type $X$ and a symplectic involution $\iota$ on $X$ are called {\it Nikulin orbifolds}.

Their deformations (in the sense of \cite{BL, MenetTorelli}) are said to be {\it orbifolds of Nikulin-type}.\end{defi}

We recall the following result by Menet.
\begin{thm}\cite{Menet}\label{thm:menet}
The second cohomology group $H^2(Y,\Z)$ of the orbifold $Y$ is endowed with a symmetric bilinear form, which is the Beauville--Bogomolov--Fujiki form $B_Y$ and thus it is a lattice. Let $B_Y$ denote also the corresponding quadratic form. Denoted by $\Sigma$ the exceptional divisor of $Y\ra X/\iota$ and by $\Delta$ the divisor induced by $\delta$ it holds:
 $$B_Y(\Sigma)=B_Y(\Delta)=-4,\ \ (\Sigma\pm\Delta)/2\in H^2(Y,\Z).$$  The lattice $\left(H^2(Y,\mathbb{Z}),B_Y\right)$ is isometric to $U(2)^{\oplus 3}\oplus E_8(-1)\oplus\langle -2\rangle\oplus\langle-2\rangle$, where the last two summands are generated by $(\Delta\pm\Sigma)/2$. 

It follows that $\Sigma$ is a class with selfintersection $-4$ and divisibility 2 in $H^2(Y,\Z)$.
\end{thm}

As a consequence of the previous theorem we get the following

\begin{cor}\label{cor:vgnonproj-orbifold}
Let $X$ be a very general fourfold of $K3^{[2]}$-type  with a symplectic involution $\iota$ such that $\NS(X)\simeq E_8(-2)$; then the corresponding Nikulin orbifold $Y$ has $\NS(Y)\simeq \langle -4\rangle$.
\end{cor}

Hence, deformations of $Y$ are not necessarily Nikulin orbifolds, since it follows from Corollary \ref{cor:vgnonproj-orbifold}  that Nikulin orbifolds are contained in a family of codimension 1.

\subsection{Families of (non necessarily projective) Nikulin orbifolds}

In Corollary \ref{cor:vgnonproj-orbifold} we describe the explicit relations between $\NS(X)$ and $\NS(Y)$ in the generic case. In the following we will consider the same problem for special subfamilies, looking for results analogous to the one of the corollary. In particular we will consider the case $X=W^{[2]}$ for a certain K3 surface $W$ and the projective case.

\begin{prop}\label{prop: non proj W^2} Let $W$ be a generic non-projective K3 surface admitting a symplectic involution $\iota_W$, i.e. $\NS(W)=E_8(-2)$. Let $X:=W^{[2]}$ be its Hilbert scheme of points and $\iota:=\iota_W^{[2]}$ be the natural involution induced by $\iota_W$. Then $\NS(X)=E_8(-2)\oplus\langle -2\rangle$, $T_X\simeq U^{\oplus 3}\oplus E_8(-2)$ and $\NS(Y)\simeq \langle -2\rangle^{\oplus 2}$, $T_Y\simeq U(2)^{\oplus 3}\oplus E_8(-1)$.  
\end{prop}
\proof By construction, the embedding of $\NS(X)$ in $H^2(X,\Z)$ is given by $\lambda_{-}(E_8(-2))\oplus \delta\simeq E_8(-2)\oplus\langle -2\rangle$. By Lemma \ref{lemma: pi_* and lambda_pm}, $\pi_*(\lambda_{-}(E_8(-2))\oplus \delta)=\pi_*(\delta)$. Since $\pi_*$ maps $\NS(X)$ to $\NS(Y)$, one deduces that $\Delta=\pi_*(\delta)=(\underline{0},\underline{0},\underline{0},\underline{0},1,1)\in U(2)^{\oplus 3}\oplus E_8(-1)\oplus \langle-2\rangle\oplus\langle-2\rangle$ is a class in $\NS(Y)$. Moreover, $\NS(Y)$ always contains the class $\Sigma=(\underline{0},\underline{0},\underline{0},\underline{0},1,-1)$ and then $\NS(Y)$ contains both $\Delta$ and $\Sigma$ and hence contains all their linear combinations which are also contained in $H^2(Y,\Z)$. In particular $\NS(Y)=\langle (\Delta+\Sigma)/2,(\Delta-\Sigma)/2\rangle\simeq \langle-2\rangle\oplus\langle-2\rangle$. The transcendental lattices are directly computed respectively as orthogonal to the N\'eron--Severi lattices inside $H^2(X,\Z)$ and $H^2(Y,\Z)$.
\endproof

\subsection{Families of projective Nikulin orbifolds and the map $\pi_*$}\label{subsec: families of projective Nikulin}

If one specializes to the projective case one has four different families of fourfolds of $K3^{[2]}$-type $X$ admitting a symplectic involution $\iota$, which depend on the chosen embedding of $\NS(X)$ in $L$ and are the ones listed in Table \eqref{table: j,NS,T}. The aim of this section is to associate to each of these families the family of Nikulin orbifolds $Y$ which are partial resolution of $X/\iota$.
The results of this section are summarized in the following table: in the first column we identify the family by choosing the embedding $\Pic(X)\subset L$; in the second column we describe the N\'eron--Severi group of $Y$, in the third its transcendental lattice and in the last we give the reference to the propositions where the results are proved.

\begin{align}\label{table: NS and T of Y}
{\renewcommand{\arraystretch}{1.4}
	\begin{array}{|c|c|c|c|c|}
		\hline
		\mbox{Embedding }\Pic(X)\subset L&\NS(Y)&T_Y&\mbox{Prop.}\\
		\hline
		j_1&\langle 4d\rangle\oplus\langle -4\rangle& U(2)^{\oplus 2}\oplus E_8(-1)\oplus\langle -4d\rangle\oplus \langle -4\rangle&\mbox{\ref{prop: j1}}\\
		\hline
		j_2,\ d\equiv 1\mod 2& \left[\begin{array}{cc}d-1&2\\2&-4\end{array}\right]& U(2)^{\oplus 2}\oplus E_7(-1)\oplus K_d(2)\oplus \langle -2\rangle&\mbox{\ref{prop: j2}}\\
		\hline
		j_3, \ d\equiv 3\mod 4&\left[\begin{array}{cc}d-1&2\\2&-4\end{array}\right]&U(2)^{\oplus 2}\oplus K_d(2)\oplus E_8(-1)&\mbox{\ref{prop: j3}}\\
		\hline
		\widetilde{j},\ d\equiv 0\mod 2&\langle d\rangle\oplus\langle -4\rangle&U^{\oplus 2}\oplus\langle -d\rangle\oplus N\oplus\langle -4\rangle&\mbox{\ref{prop: jtilde}}\\
		\hline
\end{array}}
\end{align}

To prove these results we will use the explicit embeddings described in Section \ref{subsec: 4folds with inv. lattices} and also the following explicit description of the map $\pi_*$ induced by the quotient map $\pi:X\ra X/\iota$. 

The map $$\pi_*: H^2(X,\Z)\ra H^2(X/\iota,\Z)\subset H^2(Y,\Z)$$ is compatible (as explained below) with the lattice structure induced by the Beauville--Bogomolov--Fujiki form both on $H^2(X,\Z)$ and on $H^2(Y,\Z)$. Hence we can interpret $\pi_*$ as a map between the lattices $U^{\oplus 3}\oplus E_8(-1)^{\oplus 2}\oplus\langle -2\rangle$ and $U(2)^{\oplus 3}\oplus E_8(-1)\oplus\langle -2\rangle\oplus\langle -2\rangle$. 
To describe this map we consider, as in Section \ref{subsec: 4folds with inv. lattices}, a basis of $H^2(X,\Z)$ such that $\iota^*\in O(H^2(X,\Z))$ switches the two copies of $E_8(-1)$ and acts as the identity on $U\oplus U\oplus U\oplus\langle -2\rangle$. We consider again the embeddings of the lattice $E_8(-2)$ in $E_8(-1)\oplus E_8(-1)$:
$$\begin{array}{llll}\lambda_+(b_i)&=&e_i+f_i& i=1,\ldots ,8,\\
\lambda_-(b_i)&=&e_i-f_i& i=1,\ldots ,8.\\
\end{array}$$
In particular $H^2(X,\Z)^{\iota^*}=U^{\oplus 3}\oplus\lambda_+(E_8(-2))\oplus\langle -2\rangle\simeq U^{\oplus 3}\oplus E_8(-2)\oplus\langle -2\rangle$ and  $(H^2(X,\Z)^{\iota^*})^{\perp}=\lambda_-(E_8(-2))\simeq E_8(-2)$.

Take $\underline{u}$, $\underline{v}, \underline{w}$ vectors in $U$ and $\underline{x}$, $\underline{y}$ vectors in $E_8(-1)$; for ease of notation, we will denote by $k\in\Z$ an element of $\langle -2\rangle$, referring to the $k$-th multiple of its generator (depending on the lattice this will be either $\delta$, $(\Delta+\Sigma)/2$ or $(\Delta-\Sigma)/2 $). Thus $(\underline{u},\underline{w},\underline{v},\underline{x},\underline{y},k) $ is a vector in $U^{\oplus 3}\oplus E_8(-1)^{\oplus 2}\oplus\langle -2\rangle$. Then 
\begin{equation}\label{eq: pi_*}
\pi_*(\underline{u},\underline{w},\underline{v},\underline{x},\underline{y},k)=(\underline{u},\underline{w},\underline{v},\underline{x}+\underline{y},k,k)\in U(2)^{\oplus 3}\oplus E_8(-1)\oplus\langle-2\rangle\oplus\langle-2\rangle.
\end{equation}

Hence the restriction of $\pi_*$ to $U^{\oplus 3}$ acts as the identity on the vector space, but the form is multiplied by 2; the restriction of $\pi_*$ to $E_8(-1)^{\oplus 2}$ acts as the sum of the two components on the vector space and divides the form by 2 in the quotient. 
\begin{lem}\label{lemma: pi_* and lambda_pm}
	One has: $\pi_*(\lambda_{-}(E_8(-2)))$ is trivial; $\pi_*(\lambda_{+}(E_8(-2)))=E_8(-1)$; $$\pi_*(H^2(X,\Z)^{\iota^*})=U(2)^{\oplus 3}\oplus E_8(-1)\oplus \langle -4\rangle.$$
\end{lem}
\proof It suffices to compute a basis of the sub-lattices $\lambda_{-}(E_8(-2))$, $\lambda_{-}(E_8(-2))$, $H^2(X,\Z)^{\iota^*}$ of $H^2(X,\Z)$ and then to apply the map $\pi_*$ as given in \eqref{eq: pi_*}.\endproof

\begin{prop}\label{prop: j1}
Let $d$ be a positive integer and $X_1$ be a $(\Lambda_{2d}, j_1)$-polarized fourfold of $K3^{[2]}$-type.
The fourfold $X_1$ admits a symplectic involution $\iota$ and, denoted by $Y_1$ the partial resolution of $X_1/\iota$, one has $\NS(Y_1)\simeq \langle 4d\rangle\oplus\langle -4\rangle$ and $T_{Y_1}\simeq \langle -4d\rangle\oplus U(2)^{\oplus 2}\oplus E_8(-1)\oplus \langle -4\rangle$.
\end{prop}
\proof By Proposition \ref{prop: j1 embedding} one can choose the embedding $j_1$ such that  ${j_1}_{|E_8(-2)}=\lambda_-$ and  $j_1(h):=\left(\left(\begin{array}{c}1\\d\end{array}\right),\left(\begin{array}{c}0\\0\end{array}\right), \left(\begin{array}{c}0\\0\end{array}\right), \underline{0}, \underline{0},0\right)$.
Since $\pi_*(\NS(X_1))\subset \NS(Y_1)$, one first considers $\pi_*(\NS(X_1))=\pi_*\left(\left(j_1,\lambda_{-}\right)\left(\langle 2d\rangle\oplus E_8(-2)\right)\right)=\pi_*(j_1(h))$ (where the last identity is due to Lemma \ref{lemma: pi_* and lambda_pm}). By \eqref{eq: pi_*},
$$\pi_*(j_1(h))=\left(\left(\begin{array}{c}1\\d\end{array}\right),\left(\begin{array}{c}0\\0\end{array}\right), \left(\begin{array}{c}0\\0\end{array}\right), \underline{0}, 0,0\right)\in U(2)^3\oplus E_8(-1)\oplus \langle -2\rangle^{\oplus 2},$$ so $B_Y(\pi_*(j_1(h)))=4d$. Moreover, the class $$\Sigma=\left(\left(\begin{array}{c}0\\0\end{array}\right),\left(\begin{array}{c}0\\0\end{array}\right), \left(\begin{array}{c}0\\0\end{array}\right), \underline{0}, 1 ,-1\right)$$ is contained in $\NS(Y_1)$. Hence $\NS(Y_1)$ is spanned by $\pi_*(j_1(h))$ and $\Sigma$ and there are no linear combinations with rational non integer coefficients of these classes which are also contained in $H^2(Y_1,\Z)$. So $\NS(Y_1)=\langle\pi_*(j_1(h)), \Sigma \rangle\simeq \langle 4d\rangle\oplus\langle -2\rangle$. By definition $T_{Y_1}$ is the orthogonal of $\NS(Y_1)$ in $H^2(Y_1,\Z)$. So $T_{Y_1}\simeq \langle -4d\rangle\oplus U(2)\oplus U(2)\oplus E_8(-1)\oplus \langle -4\rangle$.\endproof

\begin{prop}\label{prop: j2}
	Let $d$ be an odd positive integer and $X_2$ be a $(\Lambda_{2d}, j_2)$-polarized fourfold of $K3^{[2]}$-type.
	The fourfold $X_2$ admits a symplectic involution $\iota$ and, denoted by $Y_2$ the partial resolution of $X_2/\iota$, one has $\NS(Y_2)\simeq H_d(2):= \left[\begin{array}{cc}d-1&2\\2&-4\end{array}\right]$ and $T_{Y_2}\simeq U(2)^{\oplus 2}\oplus E_7(-1)\oplus K_d(2)\oplus \langle -2\rangle$.
\end{prop}
\proof 

By Proposition \ref{prop: j2 embedding} one can choose the embedding $j_2$ such that  ${j_2}_{|E_8(-2)}=\lambda_-$ and 
$j_2(h):=\left\{\begin{array}{ll}\left(\left(\begin{array}{c}2\\2k+2\end{array}\right),\left(\begin{array}{c}0\\0\end{array}\right), \left(\begin{array}{c}0\\0\end{array}\right), \underline{e_1}, \underline{e_1},1\right)&\mbox{ if }d=4k+1\\
	\left(\left(\begin{array}{c}2\\2k+2\end{array}\right),\left(\begin{array}{c}0\\0\end{array}\right), \left(\begin{array}{c}0\\0\end{array}\right), \underline{e_1}+\underline{e_3}, \underline{e_1}+\underline{e_3},1\right)&\mbox{ if }d=4k-1.\end{array}\right.$

As above, to compute $\NS(Y_2)$ one observes that a $\Q$-basis is given by $\pi_*(j_2(h))$ and $\Sigma$. By \eqref{eq:  pi_*}, 
$\pi_*(j_2(h))=\left(\left(\begin{array}{c}2\\2k+2\end{array}\right),\left(\begin{array}{c}0\\0\end{array}\right), \left(\begin{array}{c}0\\0\end{array}\right), 2\underline{x},1,1\right)$ and $B_Y(\pi_*(j_2(h)))=4d$. The class $\pi_*(j_2(h))-\Sigma=\left(\left(\begin{array}{c}2\\2k+2\end{array}\right),\left(\begin{array}{c}0\\0\end{array}\right), \left(\begin{array}{c}0\\0\end{array}\right), 2\underline{x},0,2\right)$ is divisible by 2 in $H^2(Y_2,\Z)$, since $d+1$ is even, thus $\left(\pi_*(j_2(h))-\Sigma\right)/2\in \NS(Y_2)$.  Finally, we get $$\NS(Y_2)=\langle \left(\pi_*(j_2(h))-\Sigma\right)/2, \Sigma\rangle=\left[\begin{array}{cc}d-1&2\\2&-4\end{array}\right].$$
The transcendental lattice is the orthogonal to $\Sigma$ and $\pi_*(j_2(h))$ in $H^2(Y_2,\Z)$. A $\Q$-basis is obtained by computing the image via $\pi^*$ of the generators of $T_{X_2}$ listed above; then one observes that the only elements which are two-divisible are those of the form $(\underline{0},\underline{0},\underline{0},2\underline{w},0,0)$, and this allows to deduce a $\Z$-basis of the lattice $T_{Y_2}$, which is of discriminant $2^8d$. Direct computation now shows that $T_{Y_2}\simeq U(2)^{\oplus 2}\oplus E_7(-1)\oplus K_d(2)\oplus \langle -2\rangle$.
\endproof

\begin{prop}\label{prop: j3}
Let $d$ be a positive integer such that $d\equiv 3\mod 4$ and $X_3$ be a $(\Lambda_{2d}, j_3)$-polarized fourfold of $K3^{[2]}$-type.
The fourfold $X_3$ admits a symplectic involution $\iota$ and, denoted by $Y_3$ the partial resolution of $X_3/\iota$, one has 
$\NS(Y_3)\simeq H_d(2)$ and $T_{Y_3}\simeq U(2)^{\oplus 2}\oplus K_d(2)\oplus E_8(-1)$. 
\end{prop}
\proof

By Proposition \ref{prop: j3 embedding} one can choose the embedding $j_3$ such that ${j_3}_{|E_8(-2)}=\lambda_-$ and  $j_3(h):=\left(\left(\begin{array}{c}2\\(d+1)/2\end{array}\right),\left(\begin{array}{c}0\\0\end{array}\right), \left(\begin{array}{c}0\\0\end{array}\right), \underline{0}, \underline{0},1\right)$. 

Since both $\pi_*(j_3(h))=\left(\left(\begin{array}{c}2\\(d+1)/2\end{array}\right),\left(\begin{array}{c}0\\0\end{array}\right) \left(\begin{array}{c}0\\0\end{array}\right), \underline{0}, 1,1\right)$ and  $\left(\pi_*(j_3(h))-\Sigma\right)/2$ are  contained in $\NS(Y_3)$,   $$\NS(Y_3)=\langle (\pi_*(j_3(h))-\Sigma)/2,\Sigma\rangle\simeq \left[\begin{array}{cc}d-1&2\\2&-4\end{array}\right]$$ and $T_{Y_3}$ is its orthogonal complement inside $U(2)^{\oplus 3}\oplus E_8(-1)\oplus\langle -2\rangle^{\oplus 2}$. Hence $T_{Y_3}\simeq U(2)^{\oplus 2}\oplus K_d(2)\oplus E_8(-1)$. \endproof

\begin{prop}\label{prop: jtilde}
Let $d$ be an even positive integer. 
and $\widetilde{X}$ be a $(\widetilde{\Lambda}_{2d}, \widetilde{j})$-polarized fourfold of $K3^{[2]}$-type.
The fourfold $\widetilde{X}$ admits a symplectic involution $\iota$ and, denoted by $\widetilde{Y}$ the partial resolution of $\widetilde{X}/\iota$, one has 
$\NS(\widetilde{Y})\simeq \langle d\rangle\oplus\langle -4\rangle$ and $T_{\widetilde{Y}}\simeq U\oplus U\oplus\langle -d\rangle\oplus N\oplus\langle -4\rangle$. \end{prop}
\proof 
By Proposition \ref{prop: jtilde embedding} one can choose the embedding $\widetilde{j}$ such that  ${\widetilde{j}}_{|E_8(-2)}=\lambda_-$ and 
$\widetilde{j}(h):=\left\{\begin{array}{ll}\left(\left(\begin{array}{c}2\\2k\end{array}\right),\left(\begin{array}{c}0\\0\end{array}\right), \left(\begin{array}{c}0\\0\end{array}\right), \underline{e_1}, \underline{e_1},0\right)&\mbox{ if }d=4k-2\mbox{ and }\\
	\widetilde{j}(h):=\left(\left(\begin{array}{c}2\\2k\end{array}\right),\left(\begin{array}{c}0\\0\end{array}\right), \left(\begin{array}{c}0\\0\end{array}\right), \underline{e_1}+\underline{e_3}, \underline{e_1}+\underline{e_3},0\right)&\mbox{ if }d=4k-4.
\end{array}\right.$

Let us consider the case $d=4k-2$.
Since  $\pi_*(\widetilde{j}(h))=\left(\left(\begin{array}{c}2\\2k\end{array}\right),\left(\begin{array}{c}0\\0\end{array}\right), \left(\begin{array}{c}0\\0\end{array}\right), 2\underline{e_1}, 0,0\right)$, $\left(\pi_*(\tilde{j}(h))\right)/2\in \NS(\widetilde{Y})$ and a basis of $\NS(\widetilde{Y})$ is given by $\left(\pi_*(\tilde{j}(h))\right)/2$ and $\Sigma$. So $\NS(\widetilde{Y})=\langle d\rangle\oplus \langle -4\rangle$ and $T_{\widetilde{Y}}$ is the orthogonal complement in $U(2)^{\oplus 3}\oplus E_8(-1)\oplus\langle -2\rangle^{\oplus 2}$ to $$\left\langle \left(\left(\begin{array}{c}1\\k\end{array}\right),\left(\begin{array}{c}0\\0\end{array}\right), \left(\begin{array}{c}0\\0\end{array}\right), \underline{e_1}, 0,0\right), \left(\left(\begin{array}{c}0\\0\end{array}\right),\left(\begin{array}{c}0\\0\end{array}\right), \left(\begin{array}{c}0\\0\end{array}\right), \underline{0}, 1,-1\right)\right\rangle.$$ 

One obtains $T_{\widetilde{Y}}\simeq U\oplus U\oplus\langle -d\rangle\oplus N\oplus\langle -4\rangle$. 

The case $d=4k-4$ is analogous.\endproof

\begin{rem}\label{rem: properties divisors in NS(Y)}{\rm The classes of divisors considered in the propositions \ref{prop: j1}, \ref{prop: j2}, \ref{prop: j3}, \ref{prop: jtilde} have a geometric meaning: the class $\Sigma$ is the effective class of the exceptional divisor; the class $\pi_*(j(h))$ is a pseudoample polarization induced on $Y$ by the ample polarization $j(h)$ on $X$, is orthogonal to $\Sigma$ and its pullback via $\pi^*$ is $2j(h)$; the class $(j(h)-\Sigma)$ corresponds to a divisor which has a positive intersection with the exceptional divisor $\Sigma$ and its pullback via $\pi^*$ is $2j(h)$.}
\end{rem}

\subsection{Subfamilies  of Nikulin fourfolds and a conjecture}\label{subsec: subfamilies of Nikuline fourfolds and a conjecture}
We now study the special cases of natural symplectic involutions induced on the Hilbert scheme of two points of a K3 surface.

\begin{prop}\label{prop: W^2 proj 2d+E8}
Let $W$ be a K3 surface such that $\NS(W)=\langle 2d\rangle\oplus E_8(-2)$ and $\iota_W$ the symplectic involution acting as $-1$ on the summand $E_8(-2)$. Let $X=W^{[2]}$ be the Hilbert square on $W$ and $\iota$ be the natural symplectic involution induced by $\iota_W$. Then $\NS(W^{[2]})=\langle 2d\rangle\oplus E_8(-2)\oplus\langle -2\rangle$ and $T_{W^{[2]}}\simeq \langle -2d\rangle\oplus U\oplus U\oplus E_8(-2)$. Denoted by $Y$ the partial resolution of $X/\iota$, $\NS(Y)\simeq \langle 4d\rangle\oplus\langle -2\rangle\oplus\langle -2\rangle$ and $T_Y\simeq \langle -4d\rangle\oplus U(2)\oplus U(2)\oplus E_8(-1)$. 
\end{prop}
\proof The proof is analogous to the previous ones. One has just to observe that the embedding of $\NS(W^{[2]})$ in $H^2(X,\Z)$ is $(j_1,\lambda_{-},\id)(h,E_8(-2),\delta)$, where $j_1(h)$ is defined in Proposition \ref{prop: j1}, $\lambda_{-}(E_8(-2))$ is as above, and $\id(\delta)=(\underline{0},\underline{0},\underline{0},\underline{0},\underline{0},1)\in U^{\oplus 3}\oplus E_8(-1)^{\oplus 2}\oplus\langle -2\rangle$. Then one applies $\pi_*$ as in \eqref{eq:  pi_*} and one observes that if both $\Delta$ and $\Sigma$ are contained in $\NS(Y)$, then also $(\Delta\pm \Sigma)/2$ is contained in $\NS(Y)$.\endproof

\begin{prop}\label{prop: W^2 proj overliattice}
Let $W$ be a K3 surface such that $\NS(W)=\widetilde{\Lambda_{2d}}$ and $\iota_W$ the symplectic involution acting as $-1$ on the summand $E_8(-2)$. Let $X=W^{[2]}$ be the Hilbert square on $W$ and $\iota$ be the natural symplectic involution induced by $\iota_W$. Then $\NS(W^{[2]})=\widetilde{\Lambda_{2d}}\oplus\langle -2\rangle$ and $T_{W^{[2]}}\simeq \langle -2d\rangle\oplus U\oplus U\oplus N$. Denoted by $Y$ the partial resolution of $X/\iota$, $\NS(Y)\simeq \langle d\rangle\oplus\langle -2\rangle\oplus\langle -2\rangle$ and $T_Y\simeq \langle -d\rangle\oplus U\oplus U\oplus N$. 
\end{prop}
\proof The proof is analogous to the previous ones. One has just to observe that the embedding of $\NS(W^{[2]})$ in $H^2(X,\Z)$ is $(\widetilde{j},\lambda_{-},\id)(h,E_8(-2),\delta)$, where $\widetilde{j}(h)$ is defined in Proposition \ref{prop: jtilde}, 
$\lambda_{-}(E_8(-2))$ is as above, and $\id(\delta)=(\underline{0},\underline{0},\underline{0},\underline{0},\underline{0},1)\in U^{\oplus 3}\oplus E_8(-1)^{\oplus 2}\oplus\langle -2\rangle$. Then one applies $\pi_*$ as in \eqref{eq:  pi_*} and concludes as above.\endproof

We computed $T_Y$ for every possible embedding $j_i$. We
observe that for all the computed $T_Y$ one can embed $T_Y$ not
only in $H^2(Y,\Z)$ as we did, but also in $L_{K3}\simeq
U\oplus U\oplus U\oplus E_8(-1)\oplus E_8(-1)$. The orthogonal of
$T_Y\hookrightarrow L_{K3}$ is the N\'eron--Severi group of
a K3 whose transcendental lattice is isometric to $T_Y$. It is
useful to compute it in view of the following conjecture
\begin{conj}\label{conj}
Let $X$ be a fourfold of $K3^{[2]}$-type admitting a symplectic involution
$\sigma$, let $Y$ be the partial resolution of $X/\sigma$ as
above, let $S$ be the K3 surface contained in $\Fix_\sigma(X)$.
Then $T_Y\simeq T_S$.
\end{conj}

As a first evidence to the conjecture we observe the following

\begin{prop}
Let $W$ be a K3 surface (projective or not) admitting a symplectic involution $\iota_W$, such that $\NS(W)=E_8(-2), \Lambda_{2d}, \widetilde{\Lambda}_{2d}$. Let $X$ be $W^{[2]}$ and $\iota$ be the natural symplectic involution induced by $\iota_W$. Then Conjecture \ref{conj} holds for $X$.
\end{prop}
\proof Let us denote by $Z$ the minimal resolution of $W/\iota_W$. It is a K3 surface and its N\'eron--Severi group and transcendental lattice are determined by the ones of $W$. We will denote by $\widetilde{\Gamma_{2e}}$ the unique even overlattice of index 2 of $\langle 2e\rangle\oplus N$ where $N$ is primitively embedded. In particular one has the following relations between the N\'eron--Severi groups 
\begin{align}\label{eq: NSW NSZ}\begin{array}{lll}\NS(W)=E_8(-2)&\mbox{ if and only if }&\NS(Z)=N\\ 
\NS(W)=\langle 2d\rangle\oplus E_8(-2)&\mbox{ if and only if }&\NS(Z)=\widetilde{\Gamma}_{4e}\\ 
\NS(W)=\widetilde{\Lambda}_{2d},\ d\equiv 0\mod 2&\mbox{ if and only if }&\NS(Z)=\langle d\rangle\oplus N\\ 
\end{array}\end{align}
which correspond to the following relations between the transcendental lattices 
\begin{align}\label{eq: TW TZ}\begin{array}{lll}T_W=U^{\oplus 3}\oplus E_8(-2)&\mbox{ if and only if }&T_Z=U^{\oplus 3}\oplus N\\ 
T_W=\langle -2d\rangle\oplus U^{\oplus 2}\oplus E_8(-2)&\mbox{ if and only if }&T_Z=\langle-4d\rangle\oplus U(2)^{\oplus 2}\oplus E_8(-1)\\ 
T_W=\langle -2d\rangle\oplus U^{\oplus 2}\oplus N,\ d\equiv 0\mod 2&\mbox{ if and only if }&T_Z=\langle -d\rangle\oplus U^{\oplus 2}\oplus  N 
\end{array}\end{align}

For every fourfold of $K3^{[2]}$-type $X$ with a symplectic involution $\iota$ the fixed locus of $\iota$ consists of 28 isolated fixed points and a K3 surface $S$. If $X=W^{[2]}$ and $\iota=\iota_W^{[2]}$, then the surface $S$ is the Nikulin surface constructed as minimal resolution of $W/\iota_W$, i.e. the surface $Z$. Hence, to conclude the proof it suffices to show that, for every $W$ (and thus every $X$), one has $T_Y\simeq T_Z$. 

If $\NS(W)=E_8(-2)$, then $T_W=U^{\oplus 3}\oplus E_8(-2)$. By Proposition \ref{prop: non proj W^2}, $T_Y\simeq U^3\oplus N$ and by \eqref{eq: TW TZ} also $T_Z\simeq U^3\oplus N$.

If $\NS(W)=\langle 2d\rangle\oplus E_8(-2)$, then $T_W=U^{\oplus 2}\oplus\langle -2d\rangle\oplus E_8(-2)$. By Proposition \ref{prop: W^2 proj 2d+E8}, $T_Y\simeq \langle -4d\rangle\oplus U(2)\oplus U(2)\oplus E_8(-1)$ and by \eqref{eq: TW TZ} also $T_Z\simeq \langle-4d\rangle\oplus U(2)^{\oplus 2}\oplus E_8(-1)$. 

If $\NS(W)=\widetilde{\Lambda}_{2d}$, with $d\equiv 0\mod 2$, then $T_W=\langle -2d\rangle\oplus U^{\oplus 2}\oplus N$. By Proposition \ref{prop: W^2 proj overliattice}, $T_Y\simeq \langle -d\rangle\oplus U^{\oplus 2}\oplus N$ and by \eqref{eq: TW TZ} also $T_Z\simeq \langle -d\rangle\oplus U^{\oplus 2}\oplus N$. \endproof

We can also show that Conjecture \ref{conj} holds  for two complete families when $d=1,3$ and the embeddings of $\Lambda_{2d}$ are respectively $j_2$ and $j_3$.

\begin{prop}\label{prop: examples conjecture d=1 j2}
Let $X$ be a $(\Lambda_{2},j_2)$-polarized fourfold of $K3^{[2]}$-type and $\sigma$ the symplectic involution described in Remark \ref{rem:involution-Hilbert-schemes}. Conjecture \ref{conj} holds in this case.
\end{prop}
\begin{proof}
We must describe the fixed locus of the symplectic involution $\sigma=\iota_{S_1}^{[2]}\circ\beta$ on $X$ (see also \cite[Lemmma 5.3]{MT}). The surface $S_1$ has a model as quartic in $\mathbb{P}^3$ and its non-symplectic involution $\iota_{S_1}$ is the restriction of a projectivity of order 4, still denoted by $\iota_{S_1}$, of the ambient projective space. For any point $P\in S_1$ we consider the line $r_P:=\langle P,\iota_{S_1}(P)\rangle$. The line $r_P$ is invariant for $\iota_{S_1}$ and thus the set of intersection points $r_P\cap S_1$ is invariant for $\iota_{S_1}$, hence there exists a point $Q\in S_1$ such that $r_P\cap S_1=\{\langle P,\iota_{S_1}(P), Q,\iota_{S_1}(Q)\rangle\}$. We consider the pair of points $(P,Q)$, which corresponds to a point in $S_1^{[2]}$. This point is a fixed point of $\sigma$, indeed $\beta(P,Q)=(\iota_{S_1}(P),\iota_{S_1}(Q))$ and $\iota_{S_1}^{[2]}(\iota_{S_1}(P),\iota_{S_1}(Q))=(P,Q)$, so $\sigma(P,Q)=(P,Q)$.
We determine a fixed point of $\sigma$ for each point $P\in S_1$. Viceversa each fixed point of $\sigma$ in $S^{[2]}$ necessarily corresponds to a pair of points in $S_1$ which lye on a $\iota_{S_1}$-invariant line. So the fixed surface of $\sigma$ is parametrized by points in $S_1$ and thus it is birational to $S_1$ (birational because in order to construct $S_1^{[2]}$ we blow up a surface and it is possible, a priori, that this introduces some exceptional divisors on the fixed locus). Nevertheless the surface contained in the fixed locus of $\sigma$ is a K3 surface as $S_1$ and thus if they are birational, they are isomorphic. So the fixed surface of $\sigma$ is a surface isomorphic to $S_1$ and in particular its transcendental lattice is $T_{S_1}\simeq U^{\oplus 2}\oplus D_4(-1)\oplus\langle -2\rangle^{\oplus 6}.$ This lattice is a 2-elementary lattice with signature $(2, 12)$ and $\delta=1$, so it is isometric to any other 2-elementary lattice with these properties, in particular to $U(2)^{\oplus 2}\oplus E_7(-1)\oplus K_1(2)\oplus \langle -2\rangle$ and the conjecture holds.
\end{proof}

 In the case of $(\Lambda_6,j_3)$-polarized fourfolds, the orthogonal of $\Lambda_6$ is $T_{6,3}=U^{\oplus 2}\oplus E_8(-2)\oplus K_3$ and $j_3(h)$ is a polarization on $X$ of degree $6$ and divisibility $2$, hence $X$ is birational to the Fano variety of a smooth cubic fourfold. In fact, this is the family of Fano varieties $F(Z)$ of smooth
symmetric cubic fourfolds $Z$ carrying a symplectic involution, as
discussed in \cite[\S 7]{Cam12}. In this case, the ample polarization
$h$ of degree $6$ is of non-split type and its orthogonal
complement is $h^\perp\cong U^{\oplus 2}\oplus E_8(-1)^{\oplus
	2}\oplus A_2(-1)$; since $E_8(-2)$ has to be orthogonal to $h$, we
obtain that the orthogonal complement of $\Lambda_3$ into $L$ is
the sublattice $T_{3,3}\cong U^{\oplus 2}\oplus E_8(-2)\oplus
A_2(-1)$.

In this case the equation of the cubic fourfold can be chosen to be $$X_0^2L_0(X_2:X_3:X_4:X_5)+X_1^2L_1(X_2:X_3:X_4:X_5)+X_0X_1L_2(X_2:X_3:X_4:X_5)+$$
$$+G(X_2:X_3:X_4:X_5)=0$$
where $L_i(X_2:X_3:X_4:X_5)$ and $G(X_2:X_3:X_4:X_5)$ are homogeneous polynomials, $\deg(L_i)=1$, $\deg(G)=3$.
The symplectic involution is induced on the Fano variety by the projective transformation $$(X_0:X_1:X_2:X_3:X_4:X_5)\rightarrow (-X_0:-X_1:X_2:X_3:X_4:X_5).$$
The fixed locus consists of 28 points, in the $(+1)$-eigenspace,  and of a K3 surface, in the $(-1)$-eigenspace, which has bidegree $(2,1)$ in $\mathbb{P}^1\times V(G)$. 

\begin{prop}\label{prop:conj-Fano-case}
Let $Z,F(Z),S$ be as above. Then $T_S\simeq T_Y\simeq U(2)^{\oplus 2}\oplus K_3(2)\oplus E_8(-1)$ and Conjecture \ref{conj} holds for $F(Z)$.
\end{prop}
\begin{proof}
Since $V(G)$ is a cubic in the projective space $\mathbb{P}^3_{(X_2:X_3:X_4:X_5)}$ the K3 surface in the fixed locus is a complete intersection of two hypersurfaces of bidegree $(2,1)$ and $(0,3)$ in $\mathbb{P}^1\times\mathbb{P}^3$.  We denote by $dP_3$ the del Pezzo cubic surface defined by $V(G)$. We recall that $dP_3$ is obtained as blow up of $\mathbb{P}^2$ in six points and, denoted by $m$ the class of a line in $\mathbb{P}^2$ and by $E_i$ the exceptional divisors of the blow up, $\NS(dP_3)$ is generated (over $\Z$) by $m,E_1,\ldots E_6$. The surface $dP_3$ is embedded in $\mathbb{P}^3$ by the anticanonical linear system $H:=3m-\sum_iE_i$. So, denoted by $H$ the hyperplane section of the cubic surface $dP_3\subset\mathbb{P}^3$, one has $m=(H+\sum_iE_i)/3\in\NS(dP_3)$. 
To compute $\NS(S)$ we first observe that it is generated, at least over $\Q$, by the classes $h_1$, $h_2$, $\ell_i$, $i=1,\ldots, 6$ where $h_1$ (resp. $h_2$) is the
restriction to the surface of pullback in $\mathbb{P}^1\times
\mathbb{P}^3$ of the hyperplane section of $\mathbb{P}^1$ (resp.
$\mathbb{P}^3$) and $\ell_i$ is the pull back of to the class $E_i\in \NS(dP_3)$. The intersection properties of
these classes are the following: $h_1^2=0$, $h_1h_2=3$,
$h_1\ell_i=1$, $i=1,\ldots, 6$, $h_2^2=6$, $h_2\ell_i=2$,
$(\ell_i)^2=-2$ and $\ell_i\ell_j=0$ if $i\neq j$. 
In particular, we observe that $h_2$ is the pullback of the divisor $H\in \NS(dP_3)$ and since $(H+\sum_iE_i)/3\in\NS(dP_3)$, we obtain that $(h_2+\sum_i\ell_i)/3\in \NS(S)$ (this divisor exhibits $S$ as double cover of $\mathbb{P}^2$ and contracts the rational curves $\ell_i$ to nodes of the branch locus of the double cover). So $\{h_1, (h_2+\sum_i\ell_i)/3, \ell_i\}$ is a set of generators of $\NS(S)$. The discriminant group of this lattice is $\Z_6 \oplus(\Z_2)^{\oplus 5}$ and the discriminant form is the opposite of the one of $U(2)^{\oplus 2}\oplus A_2(-2)$ We deduce that the transcendental lattice of $S$ is $T_S\simeq U(2)^{\oplus 2}\oplus A_2(-2)\oplus E_8(-1)$. Recalling that $A_2(-1)\simeq K_3$, we obtain that $T_S\simeq U(2)^{\oplus 2}\oplus K_3(2)\oplus E_8(-1)\simeq T_Y$ (cf. Table \ref{table: NS and T of Y}). So Conjecture \ref{conj} holds in this case.
\end{proof}

The conjecture is true at least with rational coefficients, or, in other words, the transcendental lattice of the symplectic orbifold $Y$ is the same of a (possibly twisted) Fourier--Mukai partner of the fixed K3 surface.

\begin{prop}\label{prop: conj over Q}
Let $X$ be a fourfold of $K3^{[2]}$-type admitting a symplectic involution
$\sigma$, let $Y$ be the irreducible symplectic orbifold obtained as partial resolution of $X/\sigma$, as
above, and let $S$ be the K3 surface contained in $\Fix_\sigma(X)$.
Then $T_Y\otimes\Q\simeq T_S\otimes\Q$. In particular, $\rho(S)=\rho(X)-1$.
\end{prop}
\proof  Let $\nu:S\ra X$ be the embedding of the K3 surface, we consider the restriction of forms $\nu^*:H^2(X,\C)\ra H^2(S,\C)$, which gives a morphism of Hodge structures of weight two.

Let $\omega_S\in H^{2,0}(S)$ be the restriction of a symplectic form $\omega_X\in H^{2,0}(X)$, i.e. $\omega_S=\nu^*\omega_X$; since $S$ is the fixed K3 surface, this restriction is again a symplectic form on $S$, hence $\omega_X\notin \ker \nu^*$. Moreover, the rational transcendental lattice $T_X\otimes\Q$ can be defined as the smallest rational Hodge substructure of $H^2(X,\Q)$ such that $T_X\otimes \C$ contains $\omega_X$. This implies that the restriction $\nu^*_{|T_X\otimes \Q}$ is injective: indeed, both the transcendental lattice and the kernel of a morphism of Hodge structures are irreducible Hodge substructures, thus either their intersection is trivial or they coincide, which is not the case here. In the same way one observes that the image $N$ of $\nu^*_{|T_X\otimes \Q}$ is exactly $T_S\otimes \Q$: both these Hodge substructures of $H^2(S,\Q)$ are irreducible, and their intersection contains at least $\omega_S\neq 0$, thus they coincide. In the rest of the proof we denote $\nu^*:T_X\otimes\Q\ra T_S\otimes \Q$: it is an isomorphism of irreducible Hodge structures of weight two.

Let now $\widetilde{\rho}:\widetilde{X}\ra X$ be the blow-up of the fixed K3 surface $S$, $\widetilde{\Sigma}$ be the exceptional divisor of $\rho$ and let $\widetilde{\pi}:\widetilde{X}\rightarrow Y$ be the quotient by the involution induced on  $\widetilde{X}$ by $\sigma$. We use the following diagram:
$$\xymatrix{
&\widetilde{X}\ar[r]^{\widetilde{\rho}}\ar[d]^{\widetilde{\pi}}&X\ar[d]^\pi\\
	\widetilde{Y}\ar[r]	&Y\ar[r]^\rho & X/\sigma}$$

We know from \cite[Proposition 5]{Shioda} that the transcendental lattice of a smooth resolution $\widetilde{Y}$ of a quotient $X/\Gamma$, where $X$ is smooth and $\Gamma$ is a finite group, is a Hodge structure isomorphic to the $\Gamma$-invariant part of $T_X$. In our case, 
a smooth resolution $\widetilde{Y}$ of singularities of $X/\sigma$ is also a resolution of singularities for the orbifold $Y$, hence $T_Y\otimes\Q$ is isomorphic to $T_{\widetilde{Y}}\otimes\Q$ as Hodge structures. Finally we obtained an isomorphism of rational Hodge structures of weight two $T_Y\otimes\Q\cong (T_X\otimes \Q)^\sigma=T_X\otimes\Q\cong T_S\otimes \Q$, where the first and the last isomorphisms are respectively given by $\widetilde{\rho}_*\circ\widetilde{\pi}^*$ and $\nu^*$.

We now show that this isomorphism is in fact an isometry over $\Q$.  Let $\mu_{[S]}:H^2(X,\Q)\ra H^{6}(X,\Q)$ be the cup-product with $[S]$; in \cite[Proposition B.2]{Voisin-Lagrangiennes} Voisin shows that $\ker\mu_{[S]}=\ker\nu^*$ and that, as a consequence, on $\im \nu^*$ the cup-product on $S$ is induced by cup-product on $X$ via the following equality:
\[
\langle \nu^*x,\nu^*y\rangle_S=\langle \mu_{[S]}(x),y\rangle_X=x.y.[S].
\]
In our particular case, this equality holds for all $x,y\in T_X\otimes\Q$.

Denote by $\widetilde{\Sigma}$ and $\Sigma$ respectively the exceptional divisors of $\widetilde{\rho}$ and of $\rho$.
Let $\alpha,\beta\in T_Y\otimes\Q$; by \cite[Proposition 2.11]{Menet} we have $B_Y(\alpha,\beta)=-\frac{1}{8}\alpha.\beta.\Sigma^2$.
Moreover, observing that $\widetilde{\pi}^*\Sigma=2\widetilde{\Sigma}$, a standard computation in intersection theory yields:
\begin{displaymath}
	\alpha.\beta.\Sigma^2=2\widetilde{\pi}^*\alpha.\widetilde{\pi}^*\beta.\widetilde{\Sigma}^2=-2\widetilde{\rho}_*\widetilde{\pi}^*\alpha.\widetilde{\rho}_*\widetilde{\pi}^*\beta.[S]=-2\langle\nu^*\widetilde{\rho}_*\widetilde{\pi}^*\alpha,\nu^*\widetilde{\rho}_*\widetilde{\pi}^*\beta\rangle_S.
\end{displaymath}

This shows that $B_Y(\alpha,\beta)=\frac{1}{4}\langle\nu^*\widetilde{\rho}_*\widetilde{\pi}^*\alpha,\nu^*\widetilde{\rho}_*\widetilde{\pi}^*\beta\rangle_S$  for all $\alpha,\beta\in T_Y\otimes\Q$, thus $T_S\otimes\Q\simeq T_Y(4)\otimes\Q\simeq T_Y\otimes\Q$.
\endproof

\begin{rem}
The $K3$ surfaces in the fixed locus can be seen as a generalization of Nikulin surfaces as their moduli space is densely covered by families of Nikulin surfaces.
It would be interesting to study the rationality of such moduli spaces as in \cite{FV}.
\end{rem}

\section{Orbifold Riemann--Roch formula}\label{orbRR}

\subsection{Orbifold Riemann--Roch}
In order to study projective models of Nikulin orbifolds, we need to apply the theory of orbifold Riemann--Roch, as developed in \cite{Blache} and in \cite{BRZ}. We first treat the case of Nikulin orbifolds of dimension 4, and then we generalize it to very general orbifolds of Nikulin type.

We consider again the following diagram:

$$\xymatrix{
V\ar[r]^{\widetilde{\beta}}\ar[d]^q\ar@/^2pc/[rr]^{\widetilde{r}}&\widetilde{X}\ar[r]^{\widetilde{\rho}}\ar[d]^{\widetilde{\pi}}&X\ar[d]^\pi\\
	\widetilde{Y}\ar[r]^\beta\ar@/_1pc/[rr]&Y\ar[r]^\rho & X/\iota}$$
where:
\begin{itemize}
\item $X$ is a fourfold of $K3^{[2]}$-type $\iota\in\Aut(X)$ is a symplectic involution and $W\subset \Fix_\iota(X)$ be the fixed surface;
\item $Y$ is a Nikulin orbifold of dimension 4, obtained as partial resolution of $X/\iota$; $\Sigma$ is the exceptional divisor of $\rho:Y\ra X/\iota$ and $\widetilde{X}$ is the blow-up of $X$ along $W$;
\item $\widetilde{Y}$ is the total smooth resolution of $X/\iota$, and hence of $Y$, and $V$ is the blow-up of $\widetilde{X}$ in the inverse image via $\widetilde{\rho}$ of the $28$ isolated fixed points of $\iota$. Denote respectively by $E_1,\ldots,E_{28}$ and $\widetilde{E_1},\ldots,\widetilde{E_{28}}$ the exceptional divisors on $V$ and on $\widetilde{Y}$. Moreover, let $E_W$ and $\widetilde{E_W}$ be the exceptional divisors on $V$ and $\widetilde{Y}$ over $W$ and over its image in $X/\iota$ respectively. Finally, let $E$ and $\widetilde{E}$ be respectively $\sum_{i=1}^{28}E_i+E_W$ and $\sum_{i=1}^{28}\widetilde{E_i}+\widetilde{E_W}$.
\end{itemize}

\begin{lem}\label{lemma:Grothendieck-Riemann-Roch}
Let $X,\ Y,\ \widetilde{Y}$ be as described above, and let  $\nu:W\hookrightarrow X$ be the embedding of the fixed K3 surface. Then:
\[
c_1(\widetilde{Y})=\frac{1}{2}(q_*c_1(V)+\widetilde{E})=-\sum_{i=1}^{28}\widetilde{E_i},
\]
\[
c_2(\widetilde{Y})=\frac{1}{2}q_*\widetilde{r}^*(c_2(X)+\nu_*[W])+q_*(-8\sum_{i=1}^{28}E_i^2-E_W^2)+\frac{3}{2}K_{\widetilde{Y}}\widetilde{E}+2K_{\widetilde{Y}}^2.
\]
\end{lem}
\begin{proof}
The proof follows from an application of Grothendieck--Riemann--Roch formula (see \cite[Thm. 15.2]{Fulton} combined with well-known properties of smooth blow-ups (see \cite[Example 15.4.3]{Fulton}):
\[
K_{V}=3\sum_{i=1}^{28}E_i+E_W,\ c_2(V)=\widetilde{r}^*(c_2(X)+\nu_*[W])+2\sum_{i=1}^{28}E_i^2.
\]
It is a generalization of the proof of \cite[Proof of Prop. 7.2]{CGM}.
\end{proof}

\begin{thm}[Orbifold Riemann--Roch formula]\label{thm: RR formula}
Let $D$ be a $\Q$-Cartier Weil divisor on $Y$, then  
$
q^*\beta^*D=\widetilde{r}^*H+k  E_W,
$
with $H\in\Pic(X)$,  $k\in\mathbb{Z}$; let $N$ be the number of points in which the divisor $D$ fails to be Cartier. Then 
\[
\chi(Y,D)=\frac{1}{48}H^4+\frac{1}{48}H^2.c_2(X)+(\frac{1}{16}-\frac{k^2}{8})(H_{|W})^2+3-\frac{N}{16}+\frac{k^4}{4}-\frac{3k^2}{2}.
\]
\end{thm}
\begin{proof}
Since $D$ is a $\Q$-Cartier Weil divisor on $Y$, then there exists an effective divisor $\widetilde{D}\in\Pic(\widetilde{Y})$ such that $\beta^*D=\widetilde{D}+\sum_{i=1}^{28}\lambda_i\widetilde{E_i}$ with  $\lambda_i\in\Q$: $\lambda_i=\frac{1}{2}$ if $D$ fails to be Cartier in $p_i\in\mathrm{Sing}(Y)$ for $i=1,\ldots, 28$, it is zero otherwise. We have $\beta^*D.\widetilde{E_i}=0$ for all $i$. Then the orbifold Riemann--Roch formula (\cite[Theorem 3.3]{BRZ}) is

\[
\chi(Y,D)=\chi(\widetilde{Y},\widetilde{D})=\frac{1}{24}(\beta^*D)^4+\frac{1}{12}(\beta^*D)^3.c_1(\widetilde{Y})+\frac{1}{24}(\beta^*D)^2.(c_1(\widetilde{Y})^2+c_2(\widetilde{Y}))+
\]
\[
\frac{1}{24}(\beta^*D).c_1(\widetilde{Y}).c_2(\widetilde{Y})+\chi(\mathcal{O}_{\widetilde{Y}})+\sum_{i=1}^{28}\gamma_i(D),
\]
where for each singular point $p_i\in Y$ we define $\gamma_i(D)=-\frac{1}{16}$ if $D$ is not Cartier in $p_i$, $\gamma_i(D)=0$ otherwise.

It was proven in \cite{FM} that $\chi(\oo_Y)=\chi(\oo_{\bar{Y}})=3$. Moreover, it follows from $K_{\widetilde{Y}}=\sum_i\widetilde{E}_i$, as shown in Lemma \ref{lemma:Grothendieck-Riemann-Roch}, that $\beta^*D.c_1(\widetilde{Y})=0$, hence the formula above reduces to computing $(\beta^*D)^4$ and $(\beta^*D)^2.c_2(\widetilde{Y})$. Our aim is now to reduce the intersection theory on $\widetilde{Y}$ to the intersection theory on $X$.

In our situation, we have $q^*\beta^*D=\widetilde{r}^*H+k E_W$ (indeed, if there were components in the $E_i$'s, we would have $\beta^*D.\widetilde{E}_i\neq 0$). Moreover, $q^*\widetilde{E_W}=2E_W$ and $q_*E_W=\widetilde{E_W}$; hence $E_W^4=12$, since Fujiki's relation on $Y$ implies $\widetilde{E_W}^4=6\cdot 16$.

Hence we obtain the following equalities of intersection numbers in $\mathbb{Q}$, by using Lemma \ref{lemma:Grothendieck-Riemann-Roch} and the projection formula \cite[Proposition 8.3(c)]{Fulton} (see also \cite{CGM} for further details):
\[
(\beta^*D)^4=\frac{1}{2}(q^*\beta^*D)^4=\frac{1}{2}\left((\widetilde{r}^*H)^4+k^4E_W^4+6k^2(\widetilde{r}^*H)^2.E_W^2\right)=
\frac{1}{2} H^4+6k^4-3k^2(H_{|W})^2,
\]
\[
(\beta^*D)^2.q_*\widetilde{r}^*c_2(X)=\widetilde{r}^*(H^2.c_2(X))+k^2E^2_W.\widetilde{r}^*c_2(X)=H^2.c_2(X)-k^2c_2(X).[W],
\]
\[
(\beta^*D)^2.q_*\widetilde{r}^*\nu_*[W]=\widetilde{r}^*((H_{|W})^2)+k^2E^2_W.\widetilde{r}^*\nu_*[W]=(H_{|W})^2-k^2 c_2(N_{W|X}).[W],\]
\[
(\beta^*D)^2.q_*(E^2)=-\widetilde{r}^*((H_{|W})^2)+k^2E_W^4=-(H_{|W})^2+12k^2.
\]
Many equalities and vanishings of some terms in the formulas above use the following equality for $\alpha\in A_{4-n}(X)$ (easy generalization of  \cite[Lemma 1.1]{Badescu-Beltrametti}):
\[
E_W^n.\widetilde{r}^*\alpha=(-1)^{n-1}s_{n-2}(N_{W|X}).\nu^*\alpha,
\]
combined  with $\nu^*\nu_*[W]=c_2(N_{W|X}).[W]$ (see \cite[Corollary 6.3]{Fulton}) and with the results contained in \cite{Cam12}, which give $s_1(N_{W|X})=0$, $c_2(X).[W]=36$, $s_2(N_{W|X})=-c_2(N_{W|X}).[W]=-c_2(X).[W]+c_2(W)=-12$.
\end{proof}

\begin{lem}\label{lemma:intersection-of-restriction}
Let $H\in\Pic(X)$ as in Theorem \ref{thm: RR formula}; then $(H_{|W})^2=2q(H)$.
\end{lem}
\begin{proof}
This is proven in \cite[Prop. 2.24 (4)]{Menet}, once recalled that $(H_{|W})^2=-E_W^2.\widetilde{r}^*H^2$.
\end{proof}

\begin{cor}[Riemann--Roch formula for Cartier divisors on $Y$]\label{cor: RR for Cartier}
If $D\in\Pic(Y)$ then $\chi(Y,D)=\frac{1}{4}(q(D)^2+6q(D)+12)$.
\end{cor}
\begin{proof}
In this particular case, Theorem \ref{thm: RR formula} simplifies into
\[
\chi(D)=\frac{1}{48}H^4+\frac{1}{48}H^2.c_2(X)+(\frac{1}{16}-\frac{k^2}{8})(H_{|W})^2+3+\frac{k^4}{4}-\frac{3k^2}{2}
.
\]

Since $q^*\beta^*D=\widetilde{r}^*H+kE_W$, by push-pull formula \cite[proof of Proposition 2.3(c)]{Fulton} and the commutativity of the diagram above, we have $D=\frac{1}{2}\widetilde{\pi}_*\widetilde{\rho^*}H+\frac{k}{2}\Sigma$.
The statement then follows from $q(\Sigma)=-4$, $q(\widetilde{\pi}_*\widetilde{\rho^*}H)=2q(H)$ (\cite[Prop.2.9]{Menet}), Riemann--Roch formula on $X$ (\cite[Example 23.19]{GHJ}) and Lemma \ref{lemma:intersection-of-restriction}.
\end{proof}

Corollary \ref{cor: RR for Cartier} holds for all orbifolds of Nikulin type, since it is topological in nature.

Let $Y$ be a very general Nikulin orbifold, $\beta:\widetilde{Y}\rightarrow Y$ be a smooth resolution of singularities and $\widetilde{E_1},\ldots,\widetilde{E_{28}}$ the exceptional divisors on $\widetilde{Y}$. Let $L\in\Pic(Y)$ and $D\in\Pic(Y)_{\mathbb{Q}}$ be respectively a Cartier divisor and a $\mathbb{Q}$-Cartier Weil divisor. Suppose that $2D\in \Pic(Y)$, i.e. $D=\frac{m}{2}L$ with $m\in\mathbb{Z}$ odd. By definition, $\beta^*D=\frac{m}{2}\beta^*L$.

\begin{prop}\label{prop:RR for Weil divisors v.g. case}
Let $Y$ be a four-dimensional orbifold of Nikulin type such that $\Pic(Y)=\mathbb{Z}L$, and let $D=\frac{m}{2}L$ be a $\mathbb{Q}$-Cartier divisor on $Y$, $m\in\mathbb{Z}$ odd. Then\[\chi(D)=\frac{3}{8}\left(\frac{m^4}{24}q(L)^2+m^2q(L)+8\right)-\frac{N}{16}.\]
\end{prop}
\begin{proof}
By \cite[Theorem 3.3]{BRZ}, $\chi(D)=\chi(\beta^*D)-\frac{N}{16}$ as integers. Our assumptions imply that $\beta^*D=\frac{m}{2}\beta^*L$, hence $(\beta^*D)^4=\frac{m^4}{16}L^4=\frac{3m^4}{8}q(L)^2$.

Moreover, it follows from Corollary \ref{cor: RR for Cartier} that $$\frac{1}{24}(\beta^*D)^2.c_2(\widetilde{Y})=\frac{m^2}{96}(\beta^*L)^2.c_2(\widetilde{Y})=\frac{m^2}{4}(\chi(L)-3-\frac{1}{4}q(L)^2)=\frac{3m^2}{8}q(L).$$

Hence, $\chi(D)=\frac{3}{8}(\frac{m^4}{24}q(L)^2+m^2q(L)+8)-\frac{N}{16}$.
\end{proof}

\subsection{Projective models of quotients}\label{subsec: proj models of quotient}
Let $X$ be as above, with $\rho(X)=9$. Let us denote by $A$ the ample generator of the orthogonal to $E_8(-2)$ in $\NS(X)$. In particular $A$ is preserved by $\iota$.
Then the map $\varphi_{|A|}:X\ra \mathbb{P}(H^0(X,A)^{\vee})$ is such that the automorphism $\iota$ on $X$ is induced by a projective transformation on $\mathbb{P}(H^0(X,A)^{\vee})$, still denoted by $\iota$. Hence $\iota$ acts on the vector space $U:=H^0(X,A)^{\vee}$, splitting it in the direct sum $U_+\oplus U_-$ where $U_+$ and $U_-$ are the eigenspaces of the eigenvalues $+1$ and $-1$ respectively.

The fourfold $X$ projects to $\mathbb{P}(U_+)$ and $\mathbb{P}(U_-)$; since we are considering projective spaces which are invariant for $\iota$, these two projections induce maps on the quotient, i.e. they induce the two maps $X/\iota\dashrightarrow \mathbb{P}(U_+)$ and $X/\iota \dashrightarrow \mathbb{P}(U_-)$. As rational maps these extend to maps on the partial resolution $Y$, so we obtained two maps $Y\dashrightarrow \mathbb{P}(U_+)$ and $Y\dashrightarrow \mathbb{P}(U_-)$. We are interested in these maps, which essentially give the projective models of the quotient fourfold keeping trace of the construction of this fourfold as quotient of $X$.

The maps $Y\dashrightarrow \mathbb{P}(U_+)$ and
$Y\dashrightarrow \mathbb{P}(U_-)$ are of course induced by some linear systems on $Y$ and in order to find them we are looking for divisors $D$ on $Y$ such that $\tilde{\rho}_*\tilde{\pi}^*D=\pi^*\rho_*D=A$ (because the maps to $\mathbb{P}(U_\pm)$ are induced by the projections from $\mathbb{P}(H^0(X,A)^{\vee})$).

If a connected component $Z$ of the fixed locus $\Fix_\iota(X)$ of $\iota$ on $X$ is contained in one of the two eigenspaces, then the generic member of the linear system giving the projection to the other eigenspace has to pass through $Z$. Thus the corresponding divisor on $X/\iota$ is not necessarily Cartier and passes through $N$ of the 28 singular points of $X/\iota$ and possibly through the singular surface of $X/\iota$. Nevertheless, since the map is just $2:1$, we can assume that generically the divisor on $X/\iota$  passes simply through the singularities. 
Let us now consider the partial resolution $\rho:Y\ra X/\iota$. The divisor which we are considering on $X/\iota$ induces a divisor $D_1$ on $Y$. Since $\rho$ is an isomorphism outside $\Sigma$ (which is the exceptional divisor of $\rho$ mapped to the singular surface), the divisor $D_1$ passes simply through $N$ of the 28  isolated singular points of $Y$ and then fails to be Cartier on these points. Moreover, if the divisor on $X/\iota$ passes through the singular surface, then $D_1$ has a component on the exceptional divisor $\Sigma$, with multiplicity 1; otherwise it has none.

We observe that the linear system on $X$ which corresponds to one of the projections and which is not a complete linear system (since its members have to pass through a part of $\Fix_\iota(X)$) induces a complete linear system on $V$ (where all the fixed locus is blown up).

By the previous discussion we deduce that the divisors that we are looking for on $Y$ are two divisors $D_1$ and $D_2$ (each associated to one of the two projections on the two eigenspaces) such that 
$$q^*\beta^*D_i=\tilde{r}^*A+k_iE_W,\ \ \mbox{with }k_i=0,-1\mbox{ and thus }\widetilde{\rho}_*\widetilde{\pi}^*D_i=A.$$  

Exactly one between $D_1$ and $D_2$ fails to be Cartier in a specific point (indeed a specific isolated fixed point is contained in exactly one eigenspace). The same holds true for the fixed surface (it is contained in exactly one of the eigenspaces), hence $k_i=-1$ for exactly one value among $1$ and $2$ and $k_i=0$ for the other one. Indeed, if $D_i$ is orthogonal to $\Sigma$, then $\beta^*D_i$ is orthogonal to $\widetilde{E_W}$ and $q^*\beta^*D_i$ is orthogonal to $E_W$, i.e. $k_i=0$. Similarly if the intersection of $D_i$ with $\Sigma$ is non trivial, then $k_i=-1$.

So, given $X$ a generic member of a family of fourfolds of $K3^{[2]}$-type with a symplectic involution, we determine two $\mathbb{Q}$-Cartier divisors $D_1$ and $D_2$ which give two maps $\varphi_{|D_i|}:Y_i\ra \mathbb{P}^{m_i}$. In the following table we summarize the properties of $D_1$ and $D_2$ and the dimensions $m_i$ of the projective spaces target of the map $\varphi_{|D_i|}$. We choose $D_1$ to be always orthogonal to the exceptional divisor $\Sigma$ and hence $D_2$ is always the divisor meeting $\Sigma$. Hence we have also to declare the number of points where $D_i$ fails to be Cartier (and this is always denoted by $N_i$). As in the other tables, in the first column we identify the family of $X$ (ad hence of $Y$) by giving the explicit embedding of $\NS(X)$ in $L$ and in the last we give the reference to the propositions were the results are proved.

\begin{align}\label{table: proj models Y}
\begin{array}{|c|c|c|c|c|}
\hline
\mbox{Embedding }\Pic(X)\subset L&(N_1,N_2)&m_1&m_2&\mbox{Proposition}\\
\hline&&&&\\
j_1,\ d\equiv 1\mod 2&(12,16)&\frac{d^2}{4}+\frac{3d}{2}+\frac{5}{4}&\frac{d^2}{4}+d-\frac{1}{4}&\ref{thm: D1,D2, j1}\\
\hline&&&&\\
j_1,\ d\equiv 0\mod 2&(16,12)&\frac{d^2}{4}+\frac{3d}{2}+1&\frac{d^2}{4}+d&\ref{thm: D1,D2, j1}\\
\hline&&&&\\
j_2,\ d\equiv 1\mod 2&(28,0)&\frac{d^2}{4}+\frac{3d}{2}+\frac{1}{4}&\frac{d^2}{4}+d+\frac{3}{4}&\ref{thm: D1,D2, j2,j3}\\
\hline&&&&\\
j_3,\ d\equiv 3\mod 4&(28,0)&\frac{d^2}{4}+\frac{3d}{2}+\frac{1}{4}&\frac{d^2}{4}+d+\frac{3}{4}&\ref{thm: D1,D2, j2,j3}\\
\hline&&&&\\
\widetilde{j},\ d\equiv 0\mod 2&(0,28)&\frac{d^2}{4}+\frac{3d}{2}+2&\frac{d^2}{4}+d-1&\ref{thm: D1,D2, jtilde}\\
\hline
\end{array}\end{align}

\begin{prop}\label{cor: D1 and D2}
	Let $\rho(X)=9$, $A$, $D_1$ and $D_2$ be as above and $q(A)=2d$. Then both $\chi(Y,D_1)$ and $\chi(Y,D_2)$ are integer if and only if $N_i$ and $k_i$ are as in the following (up to a possible switch between $D_1$ and $D_2$):
\begin{itemize}\item if $d$ is even then 
	\begin{itemize}\item[$\bullet$] $(N_1,k_1)=(0,0)$ 
		and $(N_2,k_2)=(28,-1)$ or 
\item[$\bullet$]  $(N_1,k_1)=(16,0)$ and $(N_2,k_2)=(12,-1)$; 
	\end{itemize}
\item if $d$ is odd then 
\begin{itemize}\item[$\bullet$] $(N_1,k_1)=(28,0)$ and $(N_2,k_2)=(0,-1)$ or 
	\item[$\bullet$] $(N_1,k_1)=(12,0)$ and $(N_2,k_2)=(16,-1)$. 
\end{itemize}	
\end{itemize}
\end{prop}
\begin{proof} We recall that for a divisor $A$ on a $K3^{[2]}$-type fourfold $X$ it holds  \[\frac{1}{48}A^4+\frac{1}{4}A^2.c_2(X)+\frac{3}{2}=\frac{1}{2}\chi(A)=\frac{1}{16}(q(A)+4)(q(A)+6),\] which, combined with Theorem \ref{thm: RR formula}, gives
	$$\chi(Y,D_i)=\frac{1}{16}(q(A)+4)(q(A)+6)-\frac{3}{2}+\left(\frac{1}{16}-\frac{k_i^2}{8}\right)(A_{|W})^2+3-\frac{N_i}{16}+\frac{k_i^4}{4}-3\frac{k_i^2}{2}.$$ Now we recall that $q(A)=2d$ and, by Lemma \ref{lemma:intersection-of-restriction}, $A_{|W}^2=2q(A)=4d$, so 
	$$\chi(Y,D_i)=\frac{d^2}{4}+\frac{5}{4}d+\frac{d}{4}-\frac{k_i^2d}{2}+3-\frac{N_i}{16}+\frac{k_i^4}{4}-3\frac{k_i^2}{2}.$$
	We observe that if $k_i=0,-1$, then $|k_i|=k_i^2=k_i^4$, hence we obtain the following formula:
	$$\chi(Y,D_i)=\frac{d^2}{4}+\frac{3d}{2}-\frac{|k_i|d}{2}-\frac{N_i}{16}-\frac{5|k_i|}{4}+3.$$
Let us assume $k_1=0$ and then $k_2=-1$. 
If $d$ is even, then $\chi(Y,D_2)\in\mathbb{Z}$ forces $N_2\equiv 12\mod 16$, which implies $N_2=12$ or $N_2=28$. If $d$ is odd then $\chi(Y,D_2)\in\mathbb{Z}$ forces $N_2\equiv 0\mod 16$, which implies $N_2=0$ or $N_2=16$. 	
\end{proof}

We observe that if $N_i=0$ for a certain divisor $D_i$, then it is a Cartier divisor on $Y$. In this case $D_i$ is $\pi_*(A)$ and it is orthogonal to the exceptional divisor if $k_i=0$, it has a positive intersection with the divisor $\Sigma$ if $k_i=-1$.

\begin{lem}
The variety $Y$ is normal with terminal singularities. In particular $Y$ is a klt variety. 
\end{lem}
\begin{proof}
The variety $Y$ is smooth outside 28 points where its singularities are locally the quotient of $\mathbb{C}^4$ by an involution $g$. In particular it is an orbifold. Hence it is normal. Moreover, the local action of the automorphism $g$ is given by the diagonal matrix ${\rm diag}(-1,-1,-,1-,1)\in SL(4)$. The age of $g$ is 2, hence the singularities of $Y$ are terminal singularities (see \cite[Theorem 6.4.3]{Joyce}), and in particular the pair $(Y,0)$ is a klt pair.
\end{proof}

\begin{prop}\label{prop: chi=h^0}
	Let $X$,$A$, $D_1$ and $D_2$ be as in Corollary \ref{cor: D1 and D2}. Then $$\chi(Y,D_i)=h^{0}(Y, D_i).$$
\end{prop}
\begin{proof}
The Kawamata--Vieweg vanishing theorem holds, see \cite[Theorem 2.70]{KM}, for the variety $Y$. With the respect to the notation in {\it loc.cit.} one can assume $\Delta=0$, and $N\equiv D_i$, $i=1,2$.
It remains to prove that the $D_i$'s are nef and big divisors. Since $\rho(X)=9$ and $A$ is the generator of $E_8(-2)^{\perp}$ in $\NS(X)$, it can be assumed to be an ample divisor. In particular it is nef, so $\pi_*(A)$ is a nef divisors. Since the sign of the top self intersection of $A$ is the same as the sign of the self intersection of $\pi_*(A)$, we deduce that $\pi_*(A)$ is a nef and big divisor, by \cite[Proposition 2.61]{KM}. Moreover, $\rho^*(D_i)=\pi_*(A)$ and since the properties of being big and nef are birational invariants, we deduce that $D_i$ is nef and big.\end{proof}

In Section \ref{sec: 4-fold with involution} we associated the divisor $A$ to a certain embedding of $\NS(X)$ in $H^2(X,\Z)$, i.e. we consider $A=j(h)$ where $h$ is vector in $U^3\oplus E_8(-1)^2\oplus \langle -2\rangle$. In Propositions \ref{prop: j1}, \ref{prop: j2}, \ref{prop: j3}, \ref{prop: jtilde}, we studied the image of this divisor under the map $\pi_*$ and we determine the generators of $\NS(Y)$.
So, by comparing the conditions on $D_i$ with the N\'eron--Severi group of $Y$ computed in Section \ref{sec: Nikulin fourfold}, one obtains the following theorems.

 \begin{thm}\label{thm: D1,D2, j1}
	Let 
	 $\NS(X)\simeq (j_1,\lambda_{-})(\langle 2d\rangle\oplus E_8(-2))$ with $A$ the generator of $j_1(\langle 2d\rangle)$.
	
	Let $D_1$ and $D_2$ be $\Q$-Cartier divisors such that $2D_1=\rho^*(\pi_*(j_1(h)))\in \NS(Y)$ and $2D_2=\rho^*\pi_*(j_1(h))-\Sigma\in \NS(Y)$. Then if $d$ is even (resp. odd), $D_1$ fails to be Cartier in 16 (resp. 12) points and $D_2$ in the other 12 (resp. the other 16) points. These divisors are such that
 $\tilde{\rho}_*\tilde{\pi}^*(D_1)=\tilde{\rho}_*\tilde{\pi}^*(D_2)=A$ and $$H^0(X,A)=(\rho^{-1}\circ\pi)^*H^0(Y,D_1)\oplus (\rho^{-1}\circ\pi)^*H^0(Y,D_2).$$ 	
\end{thm} 
\begin{proof}
Let us consider the case $d$ even. The other one is similar. One first considers $\rho^*(\pi_*(j_1(h)))\in \Pic(Y)$, $\rho^*(\pi_*(j_1(h)))-\Sigma\in \Pic(Y)$. Then there exists $D_1$ and $D_2$ $\mathbb{Q}$-Cartier such that a multiple of $D_i$, denoted by $h_iD_i$ is one prescribed element in $\Pic(Y)$. We choose $h_i$ to be the minimum  among positive integers such that $h_iD_i\in \Pic(Y)$. In particular, due to the singularities of $Y$, $h_i$ is either 1 or 2. If $h_i=1$, then $D_i$ is Cartier, otherwise it is a $\mathbb{Q}$-Cartier divisor on $Y$ and it fails to be Cartier in $N_i$ points. The possibilities for the divisors $D_1$ and $D_2$ are given in Corollary \ref{cor: D1 and D2}: the divisor $D_1$ is orthogonal to $\Sigma$, hence it is characterized by $k_1=0$; then there are two possibilities for $N_1$: either $N_1=0$ or $N_1=16$. If $N_1=0$, then $D_1$ is Cartier and $h_1=1$, otherwise $D_1$ is not Cartier and $h_1=2$. The choice of one of these two possibilities determines also the properties of $D_2$, which is necessarily $\mathbb{Q}$-Cartier and not Cartier, hence $h_2$ is necessarily 2. 

If $h_1=1$, then the divisors $D_1/2$ should be well defined, but this is not the case, since the divisor $\rho^*(\pi_*(j_1(h)))/2$ is not Cartier ($\Pic(Y)$ is described in Proposition \ref{prop: j1}). We deduce that $h_1=2$, so $N_1=16$, $N_2=12$.

The map $\rho$ is the contraction of $\Sigma$ so, if $B\in \Pic(Y)$, then $\rho_*(B)$ is a multiple of the unique generator of $\Pic(X/\iota)$. Since $\pi$ is a $2:1$ map and $A$ is invariant for $\iota$, we have $\pi^*(\rho_*(h_iD_i))=2A$ for each $h_iD_i\in \Pic(Y)$ as above.
In particular we have $\pi^*(\rho_*(D_2))=\tilde{\rho}_*\tilde{\pi}^*(D_2)=A$ (since $h_2=2$) and then the sections of $D_2$ correspond to sections of $A$ which are either all invariant or all anti-invariant for the action of the involution $\iota$. So the sections of $D_2$ span a subspace of $H^0(X,A)$ which is contained (possibly coincides) either in $U_+$ or in $U_-$ where $U_{\pm}$ are the eigenspaces of $H^0(X,A)$ for the action of $\iota^*$. Similarly, the span of the sections of $h_1D_1/2$ is contained in the other eigenspace. In order to conclude that each one of $\varphi_{|h_1D_1/2|}$ and $\varphi_{|D_2|}$ is associated to one of the two projections of $X$ to $\mathbb{P}(U_{+})$ and to $\mathbb{P}(U_-)$, it suffices to prove that the space spanned by the sections of $D_2$ (resp. $h_1D_1/2$) is not just contained, but coincides with one of the eigenspaces. So it suffices to prove that $\dim\left(H^0(Y,D_2)\oplus H^0(Y, h_1D_1/2)\right)=\dim(H^0(X,A))$.

We are now able to compute $\chi(D_i)$, $i=1,2$, by Theorem \ref{thm: RR formula} and we know that $\chi(D_i)=h^0(D_i)$, by Proposition \ref{prop: chi=h^0}. Since $q(A)=2d$, one checks
	$$\begin{array}{c}\frac{1}{8}(q(A)+6)(q(A)+4)=\dim(H^0(X,A))=\dim(H^0(Y,D_1))+\dim(H^0(Y,D_2))=\\\left(\frac{d^2}{4}+\frac{3d}{2}-1+3\right)+\left(\frac{d^2}{4}+\frac{3d}{2}-\frac{d}{2}-\frac{12}{16}-\frac{5}{4}+3\right)=\frac{d^2}{2}+\frac{5d}{2}+3.\end{array}$$ Since $\tilde{\rho}_*\tilde{\pi}^*(D_i)=A$ we conclude that 
	$$H^0(X,A)=(\rho^{-1}\circ\pi)^*H^0(Y,D_1)\oplus (\rho^{-1}\circ\pi)^*H^0(Y,D_2).$$
	
	\end{proof}

\begin{thm}\label{thm: D1,D2, j2,j3}
	Let $d\equiv 1 \mod 2$, $s=2,3$ and   
	$\NS(X)\simeq (j_s,\lambda_{-})(\langle 2d\rangle\oplus E_8(-2))$. Let $A$ be the generator of $j_s(\langle 2d\rangle)$.
	Let $D_1$ be the $\Q$-Cartier divisors such that $2D_1=\rho^*(\pi_*(j_2(h)))\in \Pic(Y)$ and $D_2$ the Cartier divisor $D_2:=\left(\rho^*\pi_*(j_1(h))-\Sigma\right)/2\in \Pic(Y)$. Then $D_1$ fails to be Cartier in 28 points,
	$\tilde{\rho}_*\tilde{\pi}^*(D_1)=\tilde{\rho}_*\tilde{\pi}^*(D_2)=A$ and $$H^0(X,A)=(\rho^{-1}\circ\pi)^*H^0(Y,D_1)\oplus (\rho^{-1}\circ\pi)^*H^0(Y,D_2).$$
	\end{thm}
\begin{proof}
	The proof is similar to the previous one. One first observes that $\rho^*(\pi_*(j_2(h)))\in \Pic(Y)$ and $\left(\rho^*(\pi_*(j_2(h)))-\Sigma\right)/2\in \Pic(Y)$ by the Propositions \ref{prop: j2} and \ref{prop: j3}. Then there exists $D_1$ and $D_2$  $\mathbb{Q}$-Cartier divisors such that a multiple of $D_i$, denoted by $h_iD_i$ is one prescribed element in $\Pic(Y)$. We choose $h_i$ to be the minimum among positive integers such that $h_iD_i\in \Pic(Y)$. In particular, due to the singularities of $Y$, $h_i$ is either 1 or 2. If $h_i=1$, then $D_i$ is Cartier, otherwise it is a $\mathbb{Q}$-Cartier divisor on $Y$ and it fails to be Cartier in $N_i$ points. The possibilities for the divisors $D_1$ and $D_2$ are given in Corollary \ref{cor: D1 and D2}: the divisor $D_1$ is orthogonal to $\Sigma$, hence it is characterized by $k_1=0$; then there are two possibilities for $N_1$, which in turn determine uniquely the values of $N_2$: either $N_1=28$ and $N_2=0$ or $N_1=12$ and $N_2=16$. If $N_1=28$, then $N_2=0$ and so $D_2$ is Cartier, otherwise (if $N_1=12$), neither $D_1$ nor $D_2$ are Cartier. As in the previous proof, we have $\pi^*(\rho_*(h_iD_1))=2A$ for each $h_iD_i\in \Pic(Y)$. Since  $\pi^*\left(\rho_*\left(\frac{\rho^*\left(\pi_*\left(j_2\left(h\right)\right)
	\right)-\Sigma}{2}\right)\right)=A$, we obtain  $D_2=\left(\rho^*(\pi_*(j_2(h)))-\Sigma\right)/2$, $h_2=1$ and $D_2$ is Cartier. This implies that $N_1=28$ and $D_1$ fails to be Cartier in all the 28 singular points of $Y$. As in the previous proposition one is able to compute $\chi(D_i)$, $i=1,2$, by Theorem \ref{thm: RR formula} and we know that $\chi(D_i)=h^0(D_i)$, by Proposition \ref{prop: chi=h^0}. So, recalling that $q(A)=2d$, one can check that 
	$$\begin{array}{c}\frac{1}{8}(q(A)+6)(q(A)+4)=\dim(H^0(X,A))=\dim(H^0(Y,D_1))+\dim(H^0(Y,D_2))=\\\left(\frac{d^2}{4}+\frac{3d}{2}-\frac{28}{16}+3\right)+\left(\frac{d^2}{4}+\frac{3d}{2}-\frac{d}{2}-\frac{5}{4}+3\right)=\frac{d^2}{2}+\frac{5d}{2}+3.\end{array}$$ Since $\tilde{\rho}_*\tilde{\pi}^*(D_i)=A$ we conclude that 
	$$H^0(X,A)=(\rho^{-1}\circ\pi)^*H^0(Y,D_1)\oplus (\rho^{-1}\circ\pi)^*H^0(Y,D_2).$$
\end{proof}

\begin{rem}\label{rem: RR d=1,3, j2,3}
	When $d=1$ and $j_s=j_2$, in the case discussed in Proposition \ref{prop: examples conjecture d=1 j2}, we obtain $h^0(D_1)=h^0(D_2)=3$, respectively with $(N_1,k_1)=(28,0)$ and $(N_2,k_2)=(0,-1)$.
When $d=3$ and $j_s=j_3$, in the case of the Fano variety of a symmetric cubic discussed before Proposition \ref{prop:conj-Fano-case}, we obtain $h^0(D_1)=8$ and $h^0(D_2)=7$, respectively with $(N_1,k_1)=(28,0)$ and $(N_2,k_2)=(0,-1)$.
\end{rem}

\begin{thm}\label{thm: D1,D2, jtilde}
	Let $d\equiv 0\mod 2$ and
	$\NS(X)\simeq \widetilde{\Lambda}_{2d}$ be the primitive closure of the embedding  $(\widetilde{j},\lambda_{-})(\langle 2d\rangle\oplus E_8(-2))$ where $A$ is the generator of $\widetilde{j}(\langle 2d\rangle)$.
	
	Let $D_1$ be the Cartier divisor $D_1=\rho^*(\pi_*(\tilde{j}(h)))/2\in \Pic(Y)$ and $D_2$ be a $\mathbb{Q}$-Cartier divisor such that $2D_2=\left((\pi_*(\widetilde{j}(h)))/2-\Sigma\right)\in \Pic(Y)$. Then $D_2$ fails to be Cartier in 28 points and
	$\tilde{\rho}_*\tilde{\pi}^*(D_1)=\tilde{\rho}_*\tilde{\pi}^*(D_2)=A$ and $$H^0(X,A)=(\rho^{-1}\circ\pi)^*H^0(Y,D_1)\oplus (\rho^{-1}\circ\pi)^*H^0(Y,D_2).$$
		
\end{thm}

\begin{proof}
The proof is completely analogous to the previous ones. We omit it.
\end{proof}

We now give an example of application of the previous theorems, in particular of Theorem \ref{thm: D1,D2, j1} with $d=1$.

Proposition \ref{prop:embeddings of Lamba(2d)} shows that, when $d=1$, the lattice $\Lambda_2\cong \langle 2\rangle\oplus \langle -2\rangle^{\oplus 8}$ admits two non-isometric embeddings  inside $L=L_{K3}\oplus\langle -2\rangle$, and in particular $j_1$ with orthogonal isometric to $T_{2,1}:=U^{\oplus 2}\oplus E_8(-2)\oplus \langle -2\rangle^{\oplus 2}$.
	
An explicit construction of this family  is given in
	\cite[\S 8]{Cam12}: it is the family of smooth double EPW sextics which carry a
	symplectic involution, as it is observed in \cite[Example
	6.8]{MW}.
	
	Indeed, the very general element of this family is $X=X_{\mathbb{A}}$ a double EPW sextic, as defined in \cite{Og}, associated with a Lagrangian subspace $\mathbb{A}\in \mathbb{LG}(\bigwedge^3V)$ invariant for the action on $\bigwedge^3V$ induced by the involution $i$ of the six-dimensional vector space $V$ which has exactly four eigenvalues $+1$. The fourfold $X_{\mathbb{A}}$ is defined as a double cover of a so-called EPW sextic $Z_{\mathbb{A}}\subset \PP(V)\simeq \PP^5$, which in this case is invariant for $i$, and it carries an ample invariant class $A\in\Pic(X_{\mathbb{A}})$ of degree two; the map $\varphi_{|A|}:X_{\mathbb{A}}\rightarrow\PP^5 $ associated to $A$ factors through the double cover  $f:X_{\mathbb{A}}\rightarrow Z_{\mathbb{A}}$.
	
	As a consequence, we get two involutions induced by $i$ on $X_{\mathbb{A}}$ and we call $\sigma$ the symplectic one among the two lifts. It is proven in \cite[Prop. 19]{Cam12} that the fixed locus $\Fix_{\sigma}(X_{\mathbb{A}})$ is the union of $28$ isolated fixed points and one K3 surfaces. In fact, $12$ points are the preimages in the double cover of six points $q_1,\ldots,q_6\in\PP(V_-)$, whereas the other $16$ points lie in the intersection of the ramification of $f$ with $\PP(V_+)$. 
	
	Finally, the fixed K3 surface $S$ is the K3 surface obtained as double cover of a quadric surface $Q\subset Z_{\mathbb{A}}\cap \PP(V_+)$ ramified along its intersection with a quartic surface. The double cover endows $S$ with a non-symplectic involution and a copy of $U(2)$ is primitively embedded in $\Pic(S)$. By Proposition \ref{prop: j1}, if the conjecture holds we should have $T_S\simeq U(2)^{\oplus 2}\oplus E_8(-1)\oplus\langle-4\rangle^{\oplus 2}.$
	
	Next we look at the Nikulin fourfold $Y$ obtained as partial resolution of $X_{\mathbb{A}}/\sigma$.  Using the notation of Corollary \ref{cor: D1 and D2} and of Theorem \ref{thm: D1,D2, j1}, we obtain on $Y$ two Weil divisors $D_1$ and $D_2$ with $(N_1,k_1)=(12,0)$ and $(N_2,k_2)=(16, -1)$.
	The orbifold Riemann--Roch formula in Theorem \ref{thm: RR formula} implies $h^0(D_1)=4$ and $h^0(D_2)=2$ (cfr. also with Table \ref{table: proj models Y}).
	
	The quotient of $\PP^5$ by the involution is the join in $\PP^{12}$ of a conic $C\subset \PP^2_1$ and a second Veronese $v_2(\PP^3)\subset \PP^9_2$ where $\PP_1^2$ an $\PP^{9}_2$ are general linear subspaces of $\PP^{12}$. With the notation as in Theorem \ref{thm: D1,D2,  j1} we have a polarization $2D_1$ on $Y$ such that $q(2D_1)=4$ (see the proof of Theorem \ref{thm: D1,D2, j1}).
	
	\begin{lem} The image of $\varphi_{|2D_1|}(Y)= \bar{Z}\subset \PP^{12}$ can be seen as the intersection
		of $J(C,v_2(\PP^3))$ with a special cubic $D$. The map $\varphi_{|2D_1|}$ is generically $2:1$ ramified along a surface.
	\end{lem}\begin{proof}
		The image $\bar{Z}$ is singular along $6$ points $C\cap D$ and three surfaces $D\cap v_2(\PP^3)\subset \PP^9_2$ (two of the components are quadric surfaces, one is a Kummer quartic)
		and the image of the singular surface of degree $40$ on $Z_{\mathbb{A}}$. Only the Kummer quartic is in the ramification since the quadric component is in the the ramification of the symplectic involution.
	\end{proof}
	
	Note that $J(C,v_2(\PP^3))\subset \PP^2$ can be seen as the intersection of the cone $C(\PP^2_1,v_2(\PP^3))$ with a quadric cone with vertex $\PP^9_2$. 
	For general projective models of a general deformation we expect the above quadric cone to be more general.

\section{Nikulin orbifolds of degree $2$}\label{Ndeg2}
There are two types of classes of degree $2$ in the Picard lattice  $U(2)^{\oplus 3}\oplus E_8(-1)\oplus\langle -2\rangle\oplus\langle-2\rangle$ of an orbifold of  Nikulin type, respectively with divisibility $1$ and $2$.

An example of a Nikulin orbifold with the class of the polarization of divisibility 2 is given by the quotient of
the Fano variety of lines on a symmetric cubic fourfold by the involution with signature $(2,4)$.  
Indeed, by Remark \ref{rem: RR d=1,3, j2,3}  the model of $X$ in $\PP^{14}$ is symmetric with an involution with invariant space $\PP^7$. After projecting from it we obtain a fourfold in $\PP^6$ being a special projective model of a Nikulin 
orbifold of degree $2$ with divisibility $2$.

In this section, we are interested  in the case of degree $2$ divisibility $1$. We first describe the special members of this complete family, given by the Nikulin orbifolds. We show that they correspond to double EPW quartics, see Lemma \ref{lemma: double EPW quartic}, and that they are double covers of complete intersections of type $(3,4)$ in $\PP^6$, see Proposition \ref{special34}. Then in Section \ref{subsec: family complete intersection 34} we generalize the previous results, showing that all the projective deformations of these Nikulin orbifolds are double covers of special complete intersections of a cubic and a quartic in $\PP^6$.

 \subsection{Geometry of $(\widetilde{\Lambda}_{4},\widetilde{j})$-polarized $K3^{[2]}$-fourfolds}(cf.~\cite[\S 3.5]{vGS})\label{jIHS}
We consider a fourfold of $K3^{[2]}$-type with Picard lattice $\Pic(X)\simeq \widetilde{\Lambda}_4$.

 Then $X$ admits a symplectic involution $\sigma$ such that the partial resolution $Y$ of $X/\sigma$ is a Nikulin orbifold with a polarization of degree 2 orthogonal to the exceptional divisor $\Sigma$, see Proposition \ref{prop: jtilde}. We denote by $D_1$ this divisor.
 
 We now describe the image $\varphi_{|D_1|}(Y)\subset \PP^6$.

 As in Proposition \ref{prop: jtilde embedding}, we can assume that $\Pic(X)$ is generated by $h$, $E_8(-2)$ with $h^2=4$ and $F_1=(h+v)/2$ where $v\in E_8(-2)$ with $v^2=-4$.
 Let $F_2=(h-v)/2$ then $F_i^2=0$ and $h=F_1+F_2$.
 After monodromy operations we can assume $h$ is big and nef.
 \begin{lem}\label{lemma: double EPW quartic} The linear system $|h|$ defines a $2:1$ map to  $C(\PP^2\times \PP^2)\subset \PP^9$.
The image is symmetric with respect to a linear involution $\sigma$ with signature $(3,7)$ on $\PP^9$ that exchanges the summands in the Segre product. Moreover, the image is isomorphic to an EPW quartic
that corresponds to a Verra threefold that is symmetric with respect to the involution exchanging the factors in $\PP^2\times \PP^2$.
 \end{lem}
 \begin{proof} By the construction of $\Pic(X)$ given in Proposition \ref{prop: jtilde embedding}, one obtains that $F_i$ are primitive in $\Pic(X)$ and that the linear system of $h=F_1+F_2$ defines a $2:1$ map to $C(\mathbb{P}^2\times\mathbb{P}^2)$, see \cite[Thm.~1.1]{IKKR}. 
 	The symplectic involution $\sigma$ acts as $-1$ on $E_8(-2)$ and then $\sigma^*F_1=F_2$. So $\sigma$ switches the two copies of $\mathbb{P}^2$ in $C(\mathbb{P}^2\times \mathbb{P}^2)$ and the $\varphi_{|h|}(X)$ is symmetric with respect to the linear involution which induces $\sigma$ and which has signature $(3,7)$ on $\PP^9$. 
 	
 	Moreover, $U(2)\simeq \langle F_1,F_2\rangle$ is primitive in $\Pic(X)$ also $E_8(-2)$.
It follows that $X$ is in the moduli space of lattice polarized hyper-K\"ahler fourfolds with $U(2)$ contained in the Picard lattice. It is thus a deformation of double EPW quartics described in \cite{IKKR}.
 
  It follows as in \cite[\S 6.5]{CKKM} that $X$ is related to a  threefold $V\subset \PP^2\times \PP^2$ symmetric with respect to the involution interchanging the factors.
   \end{proof}
 \begin{lem}\label{lemma: symmetric determinantal cubic} The quotient of $C(\PP^2\times \PP^2)\subset \PP^9$ by $\sigma$ is isomorphic to the projection of this cone from the invariant $\PP^2_-\subset \PP^9$. This quotient is a cubic hypersurface
 that is isomorphic to a cone in $\PP^6$ over a symmetric determinantal cubic fourfold in $\PP^5$.
 In particular its singular locus is a cone over the Veronese surface in $\PP^5$. 
\end{lem} 
\begin{proof} We can assume $C(\PP^2\times \PP^2)$ is defined by $2\times 2$ minors of a
 $3\times 3 $ matrix with entries being a basis of the hyperplane in ${\PP^9}^{\vee}$ orthogonal to the vertex of the cone. So elements of $\mathbb{P}^9$ can be thought as classes of pairs $(x:M)$ such that $x\in \mathbb{C}$ and $M$ is a $3\times 3$ matrix of rank 1. The involution $\sigma$ is then just the map transposing $ M$ i.e. $(x,M)\to (x,M^T)$ and  
 $$\PP^2_-=\{(0:M)| M\neq 0, \quad M+M^T=0\}.$$
 The corresponding projection is then: $\mathbb{P}^9 \ni(x,M)\to (x,M+M^T)\in \mathbb{P}^6_+$ where $\mathbb{P}^6_+=\{(x, S)| x\in \mathbb C, \quad S=S^T\}$.
 Since for a rank 1 matrix $M$, we have $M+M^T$  is  a matrix of rank at most 2, the image of the projection is a cone over the space
 symmetric matrices with trivial determinant.
 \end{proof}
We denote by $p:\mathbb{P}^9\ra \mathbb{P}^6$ the projection from the $\sigma$-invariant $\PP^2_-\subset \PP^9$ described in the previous lemma and we observe that $p$ restricts to a $2:1$ map $C(\mathbb{P}^2\times \mathbb{P}^2)\ra p(C(\mathbb{P}^2\times \mathbb{P}^2))$ with branch locus isomorphic to the cone over the diagonal of $\mathbb{P}^2\times \mathbb{P}^2$.

 \begin{prop}\label{prop: description of J} Let $J:=p(\varphi_{|h|}(X))\subset \PP^6$, then $J$ is a complete intersection $Z_3\cap T_4\subset \PP^6$ of two hypersurfaces $Z_3$ and $T_4$ of degrees 3 and 4 respectively.
 Moreover, $J$ is singular along a surface which is the disjoint union of two (possibly reducible) surfaces: $S_{16}$ of degree $16$ and $S_{36}$ of degree $36$.
 \end{prop}
 
 \begin{proof} This can be proven by considering a special example using a computer program (then using semicontinuity). Note that we run the program in positive characteristic, but since as a result of our computations both the dimension and degrees of the components are the expected ones we can conclude by semicontinuity.
 
 Let us however also describe the situation geometrically.
 Note that from the shape of the lattice of a symmetric EPW quartic we can deduce that a general symmetric EPW quartic, similarly to a general EPW quartic,  is an intersection $C(\PP^2\times \PP^2)\cap Q_4$ singular along a surface of degree $72$ see \cite[Lemma 3.2]{IKKR}. Any additional component of the singular locus would need to be related to a class contracted by the polarization. 
 
 The hypersurface $Z_3$ is the cone over the symmetric determinantal cubic hypersurface in $\mathbb{P}^5$ described in Lemma \ref{lemma: symmetric determinantal cubic} and so it is singular along a threefold generically with $A_1$ singularities. 
 Hence $J\subset \mathbb{P}^6$ is normal. Indeed, $J$ is a divisor in $Z_3$ and could be non normal only if it was singular in codimension 
 $1$. Since $p$ restricted to $C(\mathbb{P}^2\times \mathbb{P}^2)$ is $2:1$, $p_{|_{\varphi_{|h|}(X)}}$ is $2:1$ onto $J$. The branch locus is contained in the cone over the diagonal. If $J$ is singular in codimension 1, then either  $C(\PP^2\times \PP^2)\cap Q_4$ is singular in codimension $1$ or $Q_4$ contains the whole ramification locus including the vertex of the cone. Both these cases cannot occur, see \cite[Section 3]{IKKR}. We conclude that $J$ is a complete intersection of $Z_3$ with some quartic $T_4$.
 
Following the construction in \cite[Proposition 2.14]{IKKR}, we consider the varieties of $(1,1)$ conics on symmetric Verra fourfolds constructed above and we deduce that the singular surface is not contained in the cone over the diagonal, and so is not contained in the singular locus of $Z_3$. Hence the image of the singular surface of degree 72 via the projection is a surface of degree 36 being a component of the singular locus of $J$. 
 
 Summing up the (possibly reducible) components of the singular surface can be seen as:
\begin{enumerate}
\item the intersection of the singular locus of $Z_3$ with $T_4$. Since $\mathrm{Sing}(Z_3)$  is a cone over the Veronese surface, $\mathrm{Sing}(Z_3)\cap T_4$ has degree 16;  
\item the quotient of the singular locus of the symmetric EPW quartic by the involution, which is a variety of degree $72:2=36$.
\end{enumerate}

  \end{proof}
  
 \begin{prop}\label{special34} The map $\varphi_{|D_1|}:Y\ra \mathbb{P}^6$ is $2:1$ onto $\varphi_{|D_1|}(Y)\subset\mathbb{P}^6$ and its image is isomorphic to the complete intersection $Z_3\cap T_4\subset \PP^6$. The exceptional divisor $\Sigma\subset Y$ is mapped to a component of degree $4$ of the surface
  $S_{16}$. Moreover, the $(-2)$-class $D_1-\Sigma$ is effective on $Y$ and contracted to a surface via $2D_1-\Sigma$.  There are no more contractible classes on any birational model of $Y$.
  \end{prop}
  \begin{proof} 
  By Section \ref{subsec: proj models of quotient}, $\varphi_{|D_1|}(Y)$ is the image of the projection of $\varphi_{|A|}(X)$ from a $\sigma$-invariant subspace in $H^0(X,A)^{\vee}$. In our context this implies that $\varphi_{|D_1|}(Y)=p(\varphi_{|h|}(X))$, where we pose $A=h$, and hence Lemma \ref{lemma: double EPW quartic} shows that $\varphi_{|D_1|}$ is $2:1$ and  Proposition \ref{prop: description of J} describes $\varphi_{|D_1|}(Y)$.   
  
  The exceptional divisor $\Sigma$ resolves the singularity of $X/\sigma$ in the K3 surface image of the $\sigma$-fixed surface on $X$. The latter surface is in $C(\PP^2\times \PP^2)$.
  The symplectic involution on $X$  is induced by the symmetry on $\PP^9$ that interchanges the factors of $\PP^2\times \PP^2$.
  So the K3 surface in $C(\PP^2\times \PP^2)$, being fixed by the involution, is contained in the cone over the diagonal in $\PP^2\times \PP^2$. 
  It follows that its image is a component of $S_{16}$.
  By Lemma \ref{lemma:intersection-of-restriction} it is a surface $S_4$ of degree $4$ which is necessarily projectively isomorphic to the Veronese surface.
  
  For the second part we observe that the proper transform on $Y$ of the intersection of $\varphi_{|D_1|}(Y)$ with the span of $S_4$ in $\PP^6$ is the $(-2)$-class $D_1-\Sigma$. 
  The system $2D_1-\Sigma$ is big and induces its contraction since on $\varphi_{|D_1|}(Y)$ it can be seen as a system of quadrics containing the Veronese surface $S_4$ i.e. it contracts the planes spanned by conics on $S_4$ which fill the cubic $Z_3$ intersected with the span of $S_4$. The locus contracted by $2D_1-\Sigma$ is hence exactly the $(-2)$-class $D_1-\Sigma$.
  Observe that there can be no more contractible divisorial classes on any birational model of $Y$. For that, we work in codimension $1$ knowing \cite[Lemma 3.2]{MR} that any birational map is regular in codimension $1$. Now, since the Picard rank of $Y$ is $2$, among three big divisor classes one of them is a positive combination of the two other ones. In particular, if we have three divisor classes each contracted  by some map associated to a big divisor then one of these big divisors is a positive linear combination of the two remaining ones. But a positive combination of two big divisors can  only  contract subvarieties which are contracted by both divisors, so all three contracted divisors would need to have a common component. However, both $\Sigma$ and $D_1-\Sigma$ are represented by distinct irreducible effective divisors so have no common component.
  \end{proof}

In the next section we will study the complete projective family of orbifolds of Nikulin type to which $Y$ as in Proposition \ref{special34} belongs and we will show that they can all be realized as certain double covers, in complete analogy with what happens in the case of double EPW sextics. Since the full monodromy group of orbifolds of Nikulin type of dimension $4$ is not known yet, we will first use the non-symplectic involution on $Y$ given by the double cover to produce an involution of $H^2(Y,\ZZ)$ which is a monodromy operator and has the span of the divisor $D_1$ as the only invariant classes.
We recall the following notation: given an element $e\in H^2(Y,\ZZ)$, the reflection $r_e$ in $e$ is the isometry defined by $r_e(x)=x-\frac{2B_Y(x,e)}{B_Y(e,e)}e$ (it is integral only for special values of $B_Y(e,e)$ and $\mathrm{div}(e)$).
\begin{lem}\label{monodromy}
Let $D_1$ be the class with $B_Y(D_1,D_1)=2$ and divisibility 1 considered above. The isometry $-r_{D_1}$, such that $x\to -x+B_Y(x,D_1) D_1$, in $H^2(Y,\ZZ)$ is a monodromy operator of $H^2(Y,\mathbb{Z})$.
\end{lem}
 \begin{proof} The map $\varphi_{|D_1|}$ is a generically $2:1$ map onto its image and so there exists the involution $\Theta$ which is the cover of $Y\ra \varphi_{|D_1|}(Y)$. First $\varphi_{|D_1|}$ contracts the exceptional $(-4)$-class $\Sigma$ and then it identifies points switched by $i$. So $\Theta^*$ acts as $-1$ on the transcendental lattice $T_Y$ and acts trivially on $\NS(Y)$, generated by $D_1$ and $\Sigma$, see Proposition \ref{prop: jtilde}. Moreover $\Theta^*$ is a monodromy operator of $H^2(Y,\mathbb{Z})$, since it is induced by an automorphism of $Y$.
 
 Let $r_{\Sigma}$ be the reflection $r_{\Sigma}:=x\to x+\frac{1}{2}B_Y(x,\Sigma) \Sigma$. It is a monodromy operator by \cite[Proposition 1.5]{MR1}.
 We observe that 
 $-r_{D_1}=\Theta^*\circ r_{\Sigma}$ and so $-r_{D_1}$ is a monodromy operator.
 \end{proof}

 \subsection{The family of complete intersections $(3,4)$}\label{subsec: family complete intersection 34}
 Let $Y$ be an orbifold of Nikulin type of dimension four such that $$\left(H^2(Y,\Z),B_Y\right)\simeq U(2)^{\oplus 3}\oplus E_8(-1)\oplus \langle -2\rangle\oplus \langle -2\rangle,$$
 such that there exists an ample Cartier divisor $H$ on $Y$ with degree $q(H)=2$ and divisibility 1. Such an orbifold exists by surjectivity of the period map. Since the Fujiki invariant for $Y$ is $6$ we have $H^4=24$.
 \begin{thm}\label{34} The map $\varphi_{|H|}:Y\ra \mathbb{P}^6$ is $2:1$ and its image is 
a special fourfold of codimension 2 in $\PP^6$ being the complete intersection of a cubic and a quartic.
The map is branched along a surface of degree $48$.
 \end{thm}
 \begin{proof}

By Corollary \ref{cor: RR for Cartier} and the Kawamata--Viehweg vanishing theorem we have $$h^0(Y,\mathcal{O}(H))=7.$$ Hence the target space of $\varphi_{|H|}$ is $\mathbb{P}^6$.
  
  Let $Y_0$ be the special Nikulin orbifold considered in Section \ref{jIHS} and $H_0$ be the divisor $D_1\in \NS(Y_0)$ considered in Proposition \ref{special34}. 
  
  From Proposition \ref{special34}, $\varphi_{|H_0|}$ is $2:1$ and hence there exists a involution $\Theta_0$ on $Y_0$, which is the cover involution and it is non-symplectic. Moreover the image $\varphi_{|H_0|}(Y_0)$ is a normal complete intersection of type $(3,4)$. The idea of the proof is to show that a general deformation of $(Y_0,H_0)$ is of the same shape.
 
  Let $(\pi: \mathcal Y\to B, \mathcal H)$ be a family of polarized orbifolds of degree $2$ and divisibility 1 with central fiber $(Y_0,H_0)$ over a small disc $B\ni 0$.
 From Lemma \ref{monodromy}, $-r_{H_0}$ is a monodromy operator of $H^2(Y_0,\Z)$. 
 
  Let $t_n\subset B$, $Y_{t_n}$ be the fiber of $\pi$ over $t_n$ and $H_{t_n}$ the restriction of $\mathcal{H}$ to $Y_{t_n}$. We fix a sequence $t_n\ra 0$ such that $\Pic(Y_{t_n})=\ZZ H_{t_n}$. By parallel transport $-r_{H_{t_n}}$ is a monodromy operator of $H^2(Y_{t_n},\Z)$ and, by a standard argument using $\rho(Y_{t_n})=1$ and the global Torelli theorem (see for example \cite[Theorem 1.1]{MR}), 
  $-r_{H_{t_n}}$ lifts to an involution $\Theta_{t_n}\colon Y_{t_n}\to Y_{t_n}$. 

Arguing as in \cite[\S 2]{Og}, the limit of $\Theta_{t_n}$ is an involution on $Y_0$ and we show that it is $\Theta_0$.
 We denote the graph of $\Theta_{t_n}$ by $\Gamma_{t_n}$.  The analytic cycles $\Gamma_{t_n}$ converge (see proof of \cite[Thm.~4.3]{H}) to $\Gamma_0$
with a decomposition $\Gamma+n_i\Omega_i$
where $\Gamma$ is the graph of a birational map $Y_0\dasharrow Y_0$
and $\Omega_i$ are irreducible in $D_i\times E_i$ with $D_i,E_i\subset Y_0$ proper subsets.
As in \cite[\S 2]{Og}, $\Gamma_0$ induces on $H^2(Y_0,\Z)$ exactly the monodromy operator  $-r_{H_0}$  via parallel transport.

Again as in {\it loc.cit.}, the invariance of $\Gamma_{t_n}$ with respect to the exchange of the two factors in $Y_0\times Y_0$ implies that $\Gamma_0$ is invariant as well and, due to the different nature of the two parts in the decomposition above, that $\Gamma$ is the graph of a birational involution.

If $D_i$ has codimension $>1$, the action on $H^2(Y_0,\ZZ)$ of $[\Omega_i]_*$ is zero, thus we assume that $D_i$ is an effective divisor in $\Pic(Y_0)=\langle H_0,\Sigma\rangle$, but this implies that the action of $\Gamma$ on $T_{Y_0}$ coincides with the action of $\Gamma_0$, i.e. it acts as $-\id$ on the transcendental lattice.
It follows from Proposition \ref{special34} that there are exactly $2$ contractible classes on $Y_0$: the $(-4)$-class $\Sigma$ and the $(-2)$-class $H_0-\Sigma$. Hence $\Sigma$ and $H_0-\Sigma$ are preserved  by any birational map, and thus also by  $[\Gamma]_*$ i.e.
\[
[\Gamma]_*(H_0)=H_0,\ [\Gamma]_*(\Sigma)=\Sigma.
\]

We conclude that since $[\Gamma]_*$ acts on $H^2(Y_0,\Z)$ preserving both $H_0$ and $\Sigma$ and acts as minus the identity on their orthogonal in $H^2(Y_0,\Z)$ it coincides with $\Theta_0^*$. In the case of orbifolds of Nikulin type of dimension four, the only automorphism acting trivially in cohomology is the identity, as shown in \cite[Proposition 8.1]{MR1}, hence the birational involution associated to $\Gamma$ is exactly the non-symplectic involution $\Theta_0\in\Aut(Y_0)$.

We thus have a sequence $(Y_{t_n},H_{t_n}, \Theta_{t_n})$ of polarized orbifolds of Nikulin type each equipped with an involution $\Theta_{t_n}$ preserving $H_{t_n}$ and such that $(Y_0, H_0, \Theta_0)$ is its limit in the sense above.
The involutions $\Theta_{t_n}$ induce a sequence of involutions on $H^0(Y_{t_n},H_{t_n})=H^0(Y_0,H_0)$ whose limit is the map induced by $\Theta_0$ on $H^0(Y_0,H_0)$. The latter is the identity map because $\Theta_0$ is the cover involution of $\varphi_{|H_0|}$. It follows that for $n>>1$ the action of $\Theta_{t_n}$ on $H^0(Y_{t_n},H_{t_n})$  is also trivial and hence $\varphi_{|H_{t_n}|}$ is 2:1 for $n>>1$. We conclude that for general $(Y_t,H_t)$  in a neighbourhood of $(Y_0,H_0)$ the map $\varphi_{|H_t|}$ is $2:1$ onto the image contained in $\mathbb{P}^6$.  
 
We saw that $J_0:=\varphi_{|H_0|}(Y_0)$ is normal, hence
by the openness of normality the image $J_{t}$ of $Y_t$ through $|H_t|$ is also normal of codimension $2$ in $\PP^6$. Thus $J_t$ is necessarily the quotient of $Y_t$ through the involution $\Theta_t$. In particular, $J_t$ has ODP points along a surface that is smooth outside the $28$ orbifold points.
 Let us show that the general $J_t$ is also  a complete intersection,
 We consider the family $\{G_t\}_{t\in \Delta}$, with $\Delta$ a small disc, with $G_t=\varphi_{H_t}^{-1}( H_1\cap H_2)$ and $H_i$ being two chosen general hyperplanes in $\PP^6$. Note that $G_0$ is smooth and maps via $\varphi_{H_0}$ to $J_0\cap H_1\cap H_2$ which is a complete intersection $(3,4)$ in $\PP^4$ and which must admit only nodes as singularities. 
 It follows that the general $G_t$ maps via $\varphi_{H_t}$ to a nodal surface in $R_t=J_t\cap H_1\cap H_2\subset \PP^4=H_1\cap H_2$ being the quotient of $G_t$
 through an involution. Such surface is of degree $12$ and half-canonical i.e.~$K_{R_t}=2H$ (where $H$ is the hyperplane from $\PP^4$). 
 
 We can now mimic \cite[Prop.~1.2]{DPPS} to prove that $R_t$ is a complete intersection. Indeed, $R_t$ is a half canonical surface and since $R_t$ has complete intersection singularities it is a zero locus of a rank $2$ vector bundle $E$ on $\PP^4$ hence the methods of \cite[Prop.~1.2]{DPPS} apply also in this case. 
 More precisely, the case $c_1(E)^2-4c_2(E)\leq 0$ from \cite[Prop.~1.2]{DPPS} cannot occur by a generalization of the double point formula for nodal hypersurfaces proved in \cite[Thm.~5.1]{CO} ($\delta=0$ in our case). Thus $c_1(E)^2-4c_2(E)> 0$ and we conclude as in Case 2 of \cite[Prop.~1.2]{DPPS} that $R_t\subset \PP^4$ is a complete intersection $(3,4)$.
 
 We thus know that a general codimension two linear section $R_t$ of $J_t$ is a complete intersection $(3,4)$. To conclude that $J_t$ must also be such a complete intersection let us consider $U_t\supset R_t$ a general hyperplane section of $J_t$ containing $R_t$ and the exact sequences
$$0\to \mathcal{I}_{J_t}\to\mathcal{I}_{J_t}(1)\to \mathcal{I}_{U_t|\PP^5}(1)\to 0.$$
$$0\to \mathcal{I}_{U_t}\to\mathcal{I}_{U_t}(1)\to \mathcal{I}_{R_t|\PP^4}(1)\to 0.$$
 To conclude it is enough to show that $h^1(\mathcal{I}_{J_t}(2))=h^1(\mathcal{I}_{U_t}(2))=0$
 and $h^1(\mathcal{I}_{J_t}(3))=h^1(\mathcal{I}_{U_t}(3))=0$ (then the cubic and a quartic defining $R_t$ extend to the ideal of $U_t$ and then  $J_t$). But applying again the long exact sequence of cohomology the vanishing of  $h^1(\mathcal{I}_{U_t}(k))$ will follow from the vanishing of $h^2(\mathcal{I}_{U_t}(k-1))$ and $h^1(\mathcal{I}_{J_t}(k))$. It is hence enough to prove 
 \begin{equation}\label{vanishing of cohomology}
 h^2(\mathcal{I}_{J_t}(2))=h^2(\mathcal{I}_{J_t}(1))=h^1(\mathcal{I}_{J_t}(3))=h^1(\mathcal{I}_{J_t}(2))=0.
 \end{equation}
  We compute these dimensions using  the finite map  $f\colon Y_t\to J_t$: there exists a sheaf $\mathcal{F}$ on $J_t$ such that  $$f_{\ast}\mathcal{O}_{Y_t}(k)=\oo_{J_t}(k)\oplus \mathcal{F}(k).$$
 
 We get our vanishings (\ref{vanishing of cohomology}) from the fact that $h^i(\mathcal{O}_{Y_t}(k))=0$ for $i=1,2$ and $k=2,3$. We conclude that $\mathcal{I}_{J_t}$  admits a cubic and a quartic generator which, after restriction to a codimension 2 linear space, define a complete intersection. Since $J_t$ is of codimension 2 and degree 12 $J_t$ is a complete intersection.  
 
 Let us compute the degree of the singular surface of $J_t$ (in fact we can deduce this from the singular locus of $J_0$ finding $52-4=48$).
 If we denote by $S\subset Y$ a general intersection of two divisors in the system $H$ in $Y$ then we find that
 $K_S=2H|_S$ and $\chi(\mathcal{O}_S(nH))=12n^2-24n+20$.
 Denote by $G\subset \PP^4$ the image of $S$ being a complete intersection $(3,4)$.
 The involution given by $|H|$ cannot fix varieties of odd codimension (since the smooth locus of the orbifold has a symplectic form and the singular locus consists of isolated points). Moreover, it cannot fix smooth points, since it is a non-symplectic involution. The orbifold points are in the fixed locus otherwise they would map to non complete intersection singularities. So the ramification of the map is a surface. We find that $G$ is nodal and $S\to G$ is branched at the nodes. Let $t$ be the number of nodes. 
We shall compute $t$ by comparing Euler characteristics of appropriate sheaves on $S$ and $G$.
First observe that $\chi(\mathcal{O}_G)=16$ since $T$ is a complete intersection.
Next we consider the minimal resolution $\bar{G}$ of $G$ and the blow up $\bar{S}$ of $S$ at the pre-images of the nodes together with the induced map $f\colon \bar{S}\to \bar{G}$. We find $f_*\oo_{\bar{S}}=\oo_{\bar{T}}\oplus \oo_{\bar{G}}(L)$ where $2L$ is the sum of the exceptional divisors on $\bar{G}$.
We compare the Riemann--Roch formulas for $\bar{S}$ and $\bar{G}$ and conclude from  $2\chi(\oo_G)-\frac{t}{4}=\chi(\oo_{\bar{S}})=\chi(\oo_S)=20$ that $t=48$.
  
  \end{proof}

\subsection{A special subfamily of degree $2$}
We consider  Nikulin orbifolds of dimension 4 with a Cartier divisor of  degree $2$ and divisibility $1$ that form a codimension 2-subfamily of the complete family described in Theorem \ref{34}.

These orbifolds are constructed as quotients of $S^{[2]}$ by a natural symplectic involution $\sigma^{[2]}$ where $S$ are K3 surfaces with $\NS(S)\simeq \widetilde{\Lambda_{4}}$ and $\sigma$ is a symplectic involution on it.
The surfaces $S$ are double 
covers of a quadric $Q=\PP^1\times \PP^1$ branched along a $(2,2)$ curve $C$ that is symmetric with respect to the involution $\iota_Q$ exchanging the two factors of $Q$, see \cite{vGS}. We denote by $j$ the cover involution of $S\ra Q$ and we observe that $\iota_Q$ lifts to two involutions on $S$: a non symplectic involution $\iota$ and a symplectic involution $\sigma$. We observe that $\iota=j\circ \sigma$.

Then $\sigma$ induces a natural involution $\sigma^{[2]}$ on $S^{[2]}$ fixing $28$ points and a K3 surface $W$. 
The ample divisor of degree 4 on $S$ invariant for $\sigma^*$ induces a divisor $D$ on $S^{[2]}$ which is orthogonal to the exceptional divisor of $S^{[2]}\ra \Sym^2(S)$.

The map given by $|D|$ can be described as follows:
 $$\varphi\colon (\PP^3)^{[2]}\supset S^{[2]} \to \Sym^2(Q) \subset \Sym^2(\PP^3) \subset \PP^9.$$
The involution $\iota$ induces on $\PP^9$ a linear involution of the form $(-,-,-,+,+,+,+,+,+,+)$
so we have two invariant linear spaces $\PP^6_+$ and $\PP^2_-$.
By Proposition \ref{prop: W^2 proj overliattice} the Nikulin orbifold $S^{[2]}/\sigma^{[2]}$ admits a polarization $H$ of degree $2$ and divisibility $1$ induced by $D$.
\begin{lem} The image of the Nikulin orbifold $S^{[2]}/\sigma^{[2]}$ through the $4:1$ map given by $H$ is a special complete intersection $(2,3)$ in $\PP^6$. This fourfold is a degeneration of the family of $(3,4)$ intersections described in \ref{34}.
\end{lem}
\begin{proof} The image of the map $\varphi$ is a fourfold of degree $12$ in $\PP^9$ that can be seen
as the secant variety of the second Veronese embedding of a quadric surface.
The projection from $\PP^2_-$ is no longer $2:1$, as it can be checked on a special fiber.
The image is contained in a quadric since we find that $\varphi(S^{[2]})$ is contained in a quadric being a cone over $\PP^2_-$. We conclude knowing the degree of the fourfold in $\PP^6$.
\end{proof}
\begin{rem} One can show using computer calculations that the intersection $(2,3)$ above is singular along a surface of degree $12$.
\end{rem}

Note that the involution $\iota^{(2)}$ on $Q^{(2)}$ has two fixed surfaces: $A$ consisting of the pairs $(x,j(x))$ for 
$x \in Q$ and $B$ consisting of the pairs $(c_1,c_2)\subset C^{(2)}\subset Q^{(2)}$.
We see that the surface $W$ is mapped to $A$ and the isolated points are mapped to $B$.

 \subsection{Lagrangian type description}\label{Tate}
 Let us describe an object analogous to the Lagrangian subspace of dimension $10$ of $\bigwedge^3 \mathbb C^6$ for double EPW sextics.
  Suppose $(X,L)$ is a polarized orbifold of Nikulin type with degree $q(L)=2$ such that $|L|$ induces a finite $2:1$ morphism
 Then $|L|$ defines a $2:1$ map $f$ with image $Y$ being a $4$-dimensional variety of degree $12$ in $\PP^6$ singular along a 
surface.
 Since $f$ is finite, there exists a sheaf $\mathcal{F}$ on $Y$ such that  $$f_{\ast}\mathcal{O}_X=\oo_Y\oplus \mathcal{F}.$$
 We infer from the Riemann-Roch theorem the following table.
 
  \begin{center}\begin{tabular}{cccccccc}
$ H^4(\mathcal{F}(-3))$&$ H^4(\mathcal{F}(-2))$&
$0$&$0$&$0$&$0$&$0$\\
$0$&$0$&$0$&$0$&$0$&$0$&$0$  \\
$0$&$0$&$0$&$\mathbb{C} $&$0$&$0$&$0$\\
$0$&$0$&$0$&$0$&$0$&$0$&$0$  \\

$0$&$0$&0&$0$&$0$&$H^0(\mathcal{F}(2))$&$ H^0(\mathcal{F}(3))$
\end{tabular}
\end{center}
So we have the following symmetric Beilinson resolution.
\[ 0\to 28 \Omega^6(6)\xrightarrow{M^*} 3\Omega^5(5)\oplus \Omega^3(3) \oplus 3\Omega^1(1)\xrightarrow{M} 28 \mathcal{O}\to\mathcal{F}(3)\to 0\]

The matrix corresponding to $M$ from the Beilinson resolution is a matrix with three rows of $1$ forms,
one row of $3$-forms and three rows of $5$-forms.
Moreover, it has the property that $MM^*=0$ as matrices of $4$-forms (the product is induced by the exterior product of forms).

The choice of $M$ is thus the choice of a $28$ dimensional linear subspace in $$3\bigwedge^5 V_7\oplus 
\bigwedge^3V_7\oplus 3V_7,$$
isotropic for the product $b$ that can be seen as a kind of symplectic form:  $$b:(3V_7\oplus 
\bigwedge^3V_7\oplus 3\bigwedge^5 V_7)^2 \to \bigwedge^6V_7$$
given by the formula $$b((l_1,l_2,l_3,\alpha, w_1,w_2,w_3),(L_1,L_2,L_3,\beta, W_1,W_2,W_3))=$$ $$=L_1\wedge w_1+L_2\wedge w_2+L_3\wedge w_3+\alpha\wedge \beta+l_1\wedge W_1+l_2\wedge W_2+l_3\wedge W_3.$$

Note that the variety $Y$ being the support of $\mathcal F(3)$ appears as a degeneracy locus of such a map $M$. 
 
\begin{prob} Describe the "Lagrangian" $28$ space corresponding to Nikulin orbifolds i.e.~ quotients of fourfolds of $K3^{[2]}$-type as described in Section \ref{jIHS}. 
How to see that the degeneracy locus is contained in a cubic?
Is the moduli space of polarized orbifolds of Nikulin type of dimension $4$ and degree $2$ unirational?

\end{prob}

\section{Appendix: $\tilde{j}$ computations M2}
Let us find a projective model (3,4) in $\mathbb{P}^6$ of a fourfold from the family corresponding to
the embedding j.  
\begin{verbatim}
S=ZZ/11[z_1..z_10]
R=S[x_1,x_2,x_3,y_1,y_2,y_3] --the ring of P2 x P2
T=R[e_1,e_2,e_3,f_1,f_2,f_3, SkewCommutative=>true]-- the space wedge3 V
E1=(e_1+(x_1*y_1)*f_1+(x_1*y_2)*f_2+(x_1*y_3)*f_3)
E2=(e_2+(x_2*y_1)*f_1+(x_2*y_2)*f_2+(x_2*y_3)*f_3)
E3=(e_3+(x_3*y_1)*f_1+(x_3*y_2)*f_2+(x_3*y_3)*f_3)
 --E1*E2*E3 represent the image of the point in the cone over 
--  P^2 x P^2 with coordinates (1,((x_1,x_2,x_3),(y_1,y_2,y_3)))
B0=E1*E2*E3
B1=E1*E2*f_1
B2=E1*E2*f_2
B3=E1*E2*f_3
B4=E1*E3*f_1
B5=E1*E3*f_2
B6=E1*E3*f_3
B7=E2*E3*f_1
B8=E2*E3*f_2
B9=E2*E3*f_3
--Bi span the tangent to the Grassmannian in E1*E2*E3.
V1=e_1*e_2*e_3
V2=e_1*e_2*f_1
V3=e_1*e_2*f_2
V4=e_1*e_2*f_3
V5=e_1*e_3*f_1
V6=e_1*e_3*f_2
V7=e_1*e_3*f_3
V8=e_2*e_3*f_1
V9=e_2*e_3*f_2
V10=e_2*e_3*f_3
--Vi span the fixed Lagrangian space being the tangent to the Grassmannian in 
$e_1*e_2*e_3$.
W1=f_1*f_2*f_3
W2=f_1*f_2*e_1
W3=f_1*f_2*e_2
W4=f_1*f_2*e_3
W5=f_1*f_3*e_1
W6=f_1*f_3*e_2
W7=f_1*f_3*e_3
W8=f_2*f_3*e_1
W9=f_2*f_3*e_2
W10=f_2*f_3*e_3
--Wi span the Lagrangian space being the tangent to the Grassmannian in f_1*f_2*f_3, 
--together Vi and Fi span the whole space wedge^3 V.

VV=matrix{{V1,V2,V3,V4,V5,V6,V7,V8,V9,V10}}
WW=matrix{{W1,W2,W3,W4,W5,W6,W7,W8,W9,W10}}
WV=WW*(sub((transpose (VV))*WW, {e_1=>1, e_2=>1, e_3=>1, f_1=>1, f_2=>1, f_3=>1}))
-- this is a sign adjustment making the pairing given by the wedge product a 
--duality between the two bases of the Lagrangians VV and  WW

RP=ZZ/11[ed_1..ed_45]
L=genericSymmetricMatrix(RP, 9)
G=matrix{
    {0,0,0,0,0,0,0,0,1},
    {0,0,0,0,0,-1,0,0,0},
    {0,0,1,0,0,0,0,0,0},
    {0,0,0,0,0,0,0,-1,0},
    {0,0,0,0,1,0,0,0,0},
    {0,-1,0,0,0,0,0,0,0},
    {0,0,0,0,0,0,1,0,0},
    {0,0,0,-1,0,0,0,0,0},
    {1,0,0,0,0,0,0,0,0}}
--G represents a symmetry induced on  wedge^3 V by the symmetry of 
-- P^2 x P^2 exchanging x coordinates with y coordinates

FG=sub(L,transpose(mingens kernel transpose (coefficients (mingens ideal (L*G-G*L), 
Monomials=>vars RP))_1 *random(RP^27,RP^1)))
MM=(map(T,RP)) FG
M0=matrix{{0,0,0,0,0,0,0,0,0}}
MMM=(0|M0)||((transpose M0)|MM)
-- MMM represents a symmetric linear map between the two Lagrangians with bases give by 
--VV and WV i.e. a Lagrangian subspace in wedge^3 V passing through e_1*e_2*e_3
-- its relation with G means that it is invariant under the symmetry corresponding 
--to G WV*MMM
KK=VV+WV*MMM 
--KK is a basis of the Lagrangian space defined as the graph of MMM
P=(KK*transpose((map(T,S)) (matrix{{z_1..z_10}})))_0_0
KOP=coefficients(matrix{{P*B0, P*B1, P*B2, P*B3, P*B4, P*B5, P*B6, P*B7, P*B8, P*B9}},
Monomials=>{e_1*e_2*e_3*f_1*f_2*f_3})
-- KOP gives condition on elements of the Lagrangian spanned by KK given by by zi in the 
--basis KK to be elements of the tangent in E1*E2*E3 represented by the span of B
--the rank of these conditions will give the codimension of the intersection locus
KPP=KOP_1
FRT=diff(  (transpose matrix{{z_1..z_10}}), KPP)
GFD=minors(9,FRT);
--GFD describes the locus of tangents to $E1*E2*E3$ meeting the chosen symmetric 
--Lagrangian spanned by $KK$

degree GFD
D=ZZ/11[x_1,x_2,x_3,y_1,y_2,y_3]
GFO=(map(D,T)) GFD;
SB=ZZ/11[m,q_1..q_9]
fv=map(D,SB,matrix{{1,x_1*y_1,x_1*y_2,x_1*y_3,x_2*y_1,x_2*y_2,x_2*y_3,x_3*y_1,x_3*y_2,x
_3*y_3}})
man=preimage(fv, GFO);
--we see the EPW quartic GFD in the corresponding affine part of the cone over the Segre 
--embedding of P^2 x P^2
dim man
degree man
EPW=ideal (homogenize(gens man, m))
--we take the projective closure and get EPW the EPW quartic in P^9 that is 
--contained in the cone over 
--P^2x P^2 and is symmetric with respect to 
--the chosen involution
degree EPW
dim EPW
SFB=ZZ/11[s_1,s_2,s_3,s_4,s_5,s_6,tdt]
EPWsym=preimage(map(SB,SFB,matrix{{q_1,q_2+q_4,q_3+q_7,q_5,q_6+q_8,q_9,m}}), 
EPW);
--EPWsym is the projection from the anti-invariant locus given by the space of 
--skew-matrices
degree EPWsym
mingens EPWsym
--we get a complete intersection of a cubic and a quartic in P^6
S=singularLocus EPWsym;
IS=ideal S;
--IS represents the singular locus  of EPWsym
dim IS
degree IS
--the singular locus is a surface of degree 52 as expected, here we  possibly need to repeat the whole program
-- to get a general enough choice that will give the right number

R=QQ[x_1..x_4,y_1..y_4]
G=QQ[a_1..a_10]
F= (transpose matrix{{x_1..x_4}})* matrix{{y_1..y_4}}
P=(F+transpose F)
W=P^{0}|P^{1}_{1,2,3}|P^{2}_{2,3}|P^{3}_{3}
--W represents the map from P^3 x P^3$ to the space of symmetric 
--matrices which commutes with the exchange of variables. 
--its image is the symmetric square of P^3
BU=preimage(map(R,G, W), ideal(x_1^2+x_2^2+x_3^2+x_4^2,y_1^2+y_2^2+y_3^2+y_4^2))
--BU is the image of the symmetric square of the Fermat quadric in P^3 in the chosen coordinates
dim BU
saturate BU
GH=QQ[a_1,a_5..a_10]
mingens preimage(map(G,GH), BU)
--we project the symmetric square of the quadric from the anti-invariant
-- locus of a chosen symmetry preserving the quadric
\end{verbatim}

\end{document}